\documentclass[a4paper]{article}
\usepackage[utf8]{inputenc}
\usepackage{graphicx}
\usepackage[english]{babel}
\usepackage{a4wide}
\usepackage{amsmath}
\usepackage{amssymb}
\usepackage[backend=biber]{biblatex}
\usepackage{hyperref}
\usepackage{amsthm}
\usepackage{tikz}
\usetikzlibrary{matrix,arrows}
\usepackage{verbatim}
\usepackage{multirow,booktabs}
\usepackage{stmaryrd}

\allowdisplaybreaks[1]

\usetikzlibrary{cd}

\newtheorem{thm}[equation]{Theorem}
\newtheorem{cor}[equation]{Corollary}
\newtheorem{lem}[equation]{Lemma}
\newtheorem{prop}[equation]{Proposition}

\theoremstyle{definition}
\newtheorem{defn}[equation]{Definition}
\newtheorem{exa}[equation]{Example}
\newtheorem{rem}[equation]{Remark}
\newtheorem{nota}[equation]{Notation}

\numberwithin{equation}{section}

\newcommand{\recht}[1]{\textnormal{#1}}

\newcommand\cat\textsf

\DeclareRobustCommand{\VAN}[3]{#2}

\begin{document}

\title{The zeta function of stacks of $G$-zips and truncated Barsotti--Tate groups}
\author{Milan Lopuhaä-Zwakenberg}
\date{\today}

\maketitle

\tableofcontents

\section{Introduction}

Throughout this article, let $p$ be a prime number. Over a field $k$ of characteristic $p$, the truncated Barsotti--Tate groups of level 1 (henceforth BT${}_1$) were first classified in \cite{kraft1975}. The main examples of BT${}_1$s come from $p$-kernels $A[p]$ of abelian varieties $A$ over $k$. As such, these results (independently obtained) were used in \cite{oort2001} to define a stratification on the moduli space of polarised abelian varieties. In \cite{moonen2001} the first step was made towards generalising this to Shimura varieties of PEL type; this corresponds to Barsotti--Tate groups of level 1 with the action of a fixed semisimple $\mathbb{F}_p$-algebra and/or a polarisation. The classification of these BT${}_1$s with extra structure over an algebraically closed field $\bar{k}$ turned out to be related to the Weyl group of an associated reductive group over $\bar{k}$. These BT${}_1$s with extra structure were then generalised in \cite{moonenwedhorn2004} to so-called \emph{$F$-zips}, that generalise the linear algebra objects that arise when looking at the Dieudonné modules corresponding to BT${}_1$s. Over an algebraically closed field the classification of these $F$-zips was also related to the Weyl group of a certain reductive group. In \cite{pinkwedhornziegler2011} and \cite{pinkwedhornziegler2015} this was again generalised to so-called \emph{$\hat{G}$-zips}, taking the reductive group $\hat{G}$ as the primordial object. For certain choices of $\hat{G}$ these $\hat{G}$-zips correspond to $F$-zips with some additional structure. Again their classification over an algebraically closed field is expressed in terms of the Weyl group of $\hat{G}$.\\

These classifications suggest two possible directions for further research. First, one could try to study $\hat{G}$-zips over non-algebraically closed fields; the first step would then be to understand the classification over finite fields. Another direction would be to study BT${}_n$s for general $n$, either over finite fields or over algebraically closed fields. One may approach both these problems by looking at their moduli stacks. For a reductive group $\hat{G}$ over a finite field $k$, a cocharacter $\chi\colon \mathbb{G}_{\recht{m},k'} \rightarrow \hat{G}_{k'}$ defined over some finite extension $k'$ of $k$, and a subgroup scheme $\Theta \subset  \pi_0(\hat{G}_{k'})$ one can consider the stack $\hat{G}\textrm{-}\mathsf{Zip}^{\chi,\Theta}_{k'}$ of $\hat{G}$-zips of type $(\chi,\Theta)$ (see section \ref{sectiegzips}); it is an algebraic stack of finite type over $k'$. Similarly, for two nonnegative integers $h \geq d$ one can consider the stack $\mathsf{BT}_{n}^{h,d}$ of truncated Barsotti--Tate groups of level $n$, height $h$ and dimension $d$; this  is an algebraic stack of finite type over $\mathbb{F}_p$ (see \cite[Proposition 1.8]{wedhorn2001}). One way to study these stacks is via their \emph{zeta function}. For an algebraic stack of finite type $\mathsf{X}$ over a finite field $\mathbb{F}_q$ its zeta function is defined to be the element of $\mathbb{Q}\llbracket t \rrbracket$ given by
$$Z(\mathsf{X},t) = \recht{exp}\left(\sum_{v \geq 1} \frac{q^v}{v} \sum_{x \in [\mathsf{X}(\mathbb{F}_{q^v})]} \frac{1}{\#\recht{Aut}(x)}\right),$$
where $[\mathsf{X}(\mathbb{F}_{q^v})]$ denotes the set of isomorphism classes of the category $\mathsf{X}(\mathbb{F}_{q^v})$. The zeta function is related to the cohomology of $l$-adic sheaves on $\mathsf{X}$ (see \cite{behrend1993} and \cite{sun2012}). As a power series in $t$, it defines a meromorphic function that is defined everywhere (as a holomorphic map $\mathbb{C} \rightarrow \mathbb{P}^1(\mathbb{C})$), but it is not necessarily rational; the reason for this is that for stacks, contrary to schemes, the $l$-adic cohomology algebra is in general not finite dimensional (see \cite[7.1]{sun2012}).\\

The aim of this article is to calculate the zeta functions of stacks of the form $\hat{G}\textrm{-}\mathsf{Zip}^{\chi,\Theta}_{k'}$ and $\mathsf{BT}_{n}^{h,d}$. The results are stated below. In the formulation of Theorem \ref{maintheorem1}, the set $\Xi^{\chi,\Theta}$ classifies the set of isomorphism classes in $\hat{G}\textrm{-}\mathsf{Zip}^{\chi,\Theta}(\bar{\mathbb{F}}_q)$; this classification turns out to be related to the Weyl group of $\hat{G}$ (see Proposition~\ref{classificatie}). For a $\xi \in \Xi^{\chi,\Theta}$, let $a(\xi)$ be the dimension of the automorphism group of the corresponding object in $\hat{G}\textrm{-}\mathsf{Zip}^{\chi,\Theta}(\bar{\mathbb{F}}_q)$, and let $f(\xi)$ be the minimal integer $f$ such that this object has a model over $\mathbb{F}_{q^f}$. It turns out that $\Xi^{\chi,\Theta}$ has a natural action of $\Gamma := \recht{Gal}(\bar{\mathbb{F}}_q/\mathbb{F}_q)$, and that the functions $a$ and $f$ are $\Gamma$-invariant. In the formulation of Theorem \ref{maintheorem2} the notation is the same, applied to the group $\recht{GL}_{h,\mathbb{F}_p}$ (with suitable $\chi$ and $\Theta$).

\begin{thm} \label{maintheorem1}
Let $q_0$ be a power of $p$, and let $\hat{G}$ be a reductive group over $\mathbb{F}_{q_0}$. Let $q$ be a power of $q_0$, let $\chi\colon \mathbb{G}_{\recht{m},\mathbb{F}_q} \rightarrow \hat{G}_{\mathbb{F}_q}$ be a cocharacter, and let $\Theta$ be a subgroup scheme of $\pi_0(\recht{Cent}_{\hat{G}_{\mathbb{F}_q}}(\chi))$. Let $\Xi^{\chi,\Theta}$ and $\Gamma$ be as in section \ref{sectiegzips} and $f,a\colon \Gamma \backslash \Xi^{\chi,\Theta} \rightarrow \mathbb{Z}_{\geq 0}$ be as in Notation \ref{nota}. Then
$$Z(\hat{G}\textrm{-}\mathsf{Zip}^{\chi,\Theta}_{\mathbb{F}_q},t) = \prod_{\bar{\xi} \in \Gamma \backslash \Xi^{\chi,\Theta}} \frac{1}{1-(q^{-a(\bar{\xi})}t)^{f(\bar{\xi})}}.$$
\end{thm}

\begin{thm} \label{maintheorem2}
Let $h,n > 0$ and $0 \leq d \leq h$ be integers. Let $\Xi$ and $a\colon \Xi \rightarrow \mathbb{Z}_{\geq 0}$ be as in Notation \ref{truncnot}. Then
$$Z(\mathsf{BT}_{n}^{h,d},t) = \prod_{\xi \in \Xi} \frac{1}{1-p^{-a(\xi)}t}.$$
In particular the zeta function of the stack $\mathsf{BT}_{n}^{h,d}$ does not depend on $n$.
\end{thm}

All the terminology used in the statements above will be introduced in due time. For now let us note that the functions $a$ and $f$ can also be expressed in terms of the action of the Weyl group of $\hat{G}$ on the root system, and are readily calculated for a given $(\hat{G},\chi,\Theta)$ (see Example \ref{voorbeeld2}).

\section{The Zeta function of quotient stacks} \label{sectiezeta}

In this section, we discuss some properties of quotient stacks over finite fields and their zeta functions.

\subsection{Fields of definition in quotient stacks} \label{fod}

Let $G$ be an algebraic group over a field $k$, and let $X$ be a scheme over $k$ with a left action of $G$. Then one can consider the \emph{quotient stack} $[G \backslash X]$, whose objects over a $k$-scheme $S$ are pairs $(T,f)$ of a $G$-torsor $T$ over $S$ and a $G_S$-equivariant morphism $f\colon T \rightarrow X_S$, with the obvious notions of morphisms and pullbacks. The automorphism group of such a pair $(T,f)$ can be considered as an algebraic group; it is the subgroup $\recht{Stab}(f) \subset \recht{Aut}_G(T)$.

\begin{rem} \label{fodopmerking}
Denote by ${}_G X$ the $\recht{Gal}(\bar{k}/k)$-set $G(\bar{k}) \backslash X(\bar{k})$. For a finite field extension $k'/k$  we will occasionally use the notation ${}_G X(k') := ({}_G X)^{\recht{Gal}(\bar{k}/k')}$. Over $\bar{k}$, all torsors are trivial, hence the set of isomorphism classes $[[G\backslash X](\bar{k})]$ corresponds bijectively to ${}_G X$, by sending a pair $(G_{\bar{k}},f)$ to  the orbit $f(G) \subset X_{\bar{k}}$. If $C \in {}_G X$, we will occasionally write $P(C)$ for the corresponding isomorphism class in $[G\backslash X](\bar{k})$. Note that for any element $c \in C$ we get an isomorphism $\recht{Stab}_G(c) \cong \recht{Aut}(P(C))$ of algebraic groups over $\bar{k}$.
\end{rem}

For the next main result we first need a lemma. For an algebraic group $G$ over a field $k$, recall that the set of cocycles $\recht{Z}^1(k,G)$ is defined as the set of continuous maps $z\colon \recht{Gal}(\bar{k}/k) \rightarrow G(\bar{k})$ (where the right hand side has the discrete topology) satisfying 
\begin{equation} \label{cocyclecondition}
z(\sigma \sigma') = z(\sigma)\cdot {}^{\sigma}z(\sigma')
\end{equation}
for all $\sigma,\sigma' \in \recht{Gal}(\bar{k}/k)$.

\begin{lem} \label{cocykel}
Let $k$ be a finite field, and let $G$ by an algebraic group over $k$. Let $\sigma \in \recht{Gal}(\bar{k}/k)$ be the $\#k$-th power Frobenius endomorphism of $\bar{k}$. Then there is a natural bijection
\begin{eqnarray*}
\recht{Z}^1(k,G) &\stackrel{\sim}{\rightarrow} & G(\bar{k})\\
z &\mapsto &z(\sigma).
\end{eqnarray*}
\end{lem}

\begin{proof}
Let $\Gamma$ be the Galois group $\recht{Gal}(\bar{k}/k)$. Since $\langle \sigma \rangle \subset \Gamma$ is a dense subgroup, the map is certainly injective. To show that it is surjective, let $g \in G(\bar{k})$, and define a map $z\colon \langle \sigma \rangle \rightarrow G(\bar{k})$ by
$$z(\sigma^n) = \left\{\begin{array}{ll} g \cdot {}^\sigma g \cdots  {}^{\sigma^{n-1}}g, & \textrm{ if $n \geq 0$}; \\
{}^{\sigma^{-1}} g^{-1} \cdots {}^{\sigma^{n}}g^{-1}& \textrm{ if $n < 0$}.\end{array}\right.$$
This satisfies the cocycle condition (\ref{cocyclecondition}) on $\langle \sigma \rangle$. Let $e$ be the unit element of $G(\bar{k})$. To show that we can extend $z$ continuously to $\Gamma$, we claim that there is an integer $n$ such that $z(\sigma^N) = e$ for all $N \in n\mathbb{Z}$. To see this, let $k'$ be a finite extension of $k$ such that $g \in G(k')$. Then from the definition of the map $z$ we see that $z$ maps $\langle \sigma \rangle$ to $G(k')$. The latter is a finite group, and hence there must be $m < m' \in \mathbb{Z}_{\geq 0}$ such that $z(\sigma^m) = z(\sigma^{m'})$. Set $n = m'-m$. From the definition of $z$ we see that
$$z(\sigma^{m'}) = z(\sigma^{m}) \cdot {}^{\sigma^{m}}g \cdots {}^{\sigma^{m'-1}}g,$$
hence ${}^{\sigma^{m}}g \cdots {}^{\sigma^{m'-1}}g = e$; but the left hand side of this is equal to ${}^{\sigma^{m}}z(\sigma^{n})$, hence $z(\sigma^n) = e$. The cocycle condition (\ref{cocyclecondition}) now shows that $z(\sigma^{N}) = e$ for every multiple $N$ of $n$; furthermore, we see that for general $f \in \mathbb{Z}$ the value $z(\sigma^f)$ only depends on $\bar{f} \in \mathbb{Z}/n\mathbb{Z}$. Hence we can extend $z$ to all of $\Gamma$ via the composite map
$$\Gamma \twoheadrightarrow \Gamma/n\Gamma \stackrel{\sim}{\rightarrow} \langle \sigma \rangle / \langle \sigma^n \rangle \stackrel{z}{\rightarrow} G(\bar{k}),$$
and this is an element of $\recht{Z}^1(k,G)$ that sends $\sigma$ to $g$; hence the map in the lemma is surjective, as was to be shown.
\end{proof}

\begin{defn}
Let $k$ be a finite field, let $G$ be an algebraic group over $k$, and let $X$ be a $k$-scheme with a left action of $G$. Let $k'$ be an algebraic extension of $k$. Let $(T,f)$ be an object of $[G\backslash X](\bar{k})$. Then a \emph{model of $(T,f)$ over $k'$} is an object $(\tilde{T},\tilde{f})$ of $[G\backslash X](k')$ such that $(\tilde{T},\tilde{f})_{\bar{k}} \cong (T,f)$.
\end{defn}

\begin{prop} \label{stacksfod}
Let $k$ be a finite field, let $G$ be an algebraic group over $k$, and let $X$ be a $k$-scheme with a left action of $G$. Let $C \in {}_GX$. Suppose $k$ is finite. Then for every finite extension $k' \supset k$ the following are equivalent:
\begin{enumerate}
\item $P(C)$ admits a model over $k'$;
\item $C \in {}_GX$ is fixed under the action of $\recht{Gal}(\bar{k}/k')$.
\end{enumerate}
\end{prop}

\begin{rem}
In view of the Proposition above we may define the \emph{field of definition} of $P(C)$ to be the field of definition of $G(\bar{k})x \in G(\bar{k}) \backslash X(\bar{k})$; by the Proposition it is the minimal field extension of $k$ over which $P(C)$ has a model. Note, however, that this model is necessarily not unique. For example, if $X = \recht{Spec} \ k$, so that $[G \backslash X]$ is the classifying stack $\recht{B}G$, the models over $k$ of the unique element of $[[G \backslash X](\bar{k})]$ are precisely the $G$-torsors over $k$.
\end{rem}

\begin{proof}[Proof of Proposition \ref{stacksfod}.] Write $P(C) = (T,f)$. First suppose $(T,f)$ admits a model over $k'$. Then $f(T(\bar{k})) \subset X(\bar{k})$ is necessarily invariant under the action of $\recht{Gal}(\bar{k}/k')$. This image is equal to $C$, which proves one implication. Now consider the other implication. Over $\bar{k}$, all torsors are trivial, hence $T = G$. Let $e$ be the unit element of $G(\bar{k})$, and let $\sigma$ be the $\#k$-th power Frobenius on $\bar{k}$. By assumption the action of $\sigma$ sends the $G(\bar{k})$-orbit $C \subset X(\bar{k})$ to itself, hence there exists a $g_0 \in G(\bar{k})$ such that $g_0\cdot f(e) = \sigma(f(e))$. Furthermore, we know that the automorphism group of $G$ as a $G$-torsor is isomorphic to $G^{\recht{op}}$, where we send a $g \in G^{\recht{op}}$ to the automorphism $t \mapsto t \cdot g$ of $G$ as a left $G$-torsor. Now by Lemma \ref{cocykel} there exists a cocycle in $\recht{Z}^1(k',G^{\recht{op}})$ that sends $\sigma$ to $g_0$. This cocycle defines a twist $\tilde{T}$ of $T$ over $k'$; its Galois action is given by $\tilde{\sigma}(t) = \sigma(t)g_0$, where $\sigma$ is the regular action of $\recht{Fr}_{k'}$ on $G(\bar{k})$. Under this twist the map $f\colon \tilde{T}_{\bar{k}} \rightarrow X_{\bar{k}}$ is $\sigma$-invariant, hence defined over $k'$. This shows that this twist defines a model $(\tilde{T},\tilde{f})$ over $k'$, as was to be shown.
\end{proof}

\subsection{Counting points on quotient stacks}

As before, let $k$ be a field, let $G$ be an algebraic group over $k$, and let $X$ be a scheme over $k$ with a left action of $G$. For the rest of this section, we assume that $k = \mathbb{F}_{q}$, and we are interested in the zeta function of the stack $[G
\backslash X]$. The first main result to prove in this section is the following:

\begin{thm} \label{zetaquotient}
Let $k = \mathbb{F}_q$ be a finite field, let $G$ be an algebraic group over $k$, and let $X$ be a scheme over $k$ with a left action of $G$. Suppose that for each $C \in {}_G X$ the identity component of the group scheme $\recht{Aut}(P(C))^{\recht{red}}$ is unipotent. Then
$$Z([G\backslash X],t) = \recht{exp}\left(\sum_{v \geq 1} \frac{t^v}{v} \sum_{C \in {}_G X(\mathbb{F}_{q^v})} q^{-\recht{dim}(\recht{Aut}(P(C)))\cdot v}\right).$$
\end{thm}

\begin{rem} \label{zetaopm}
Consider a pair $(G_{\bar{k}},f)$ over $\bar{k}$. Then the group scheme of $G$-torsor automorphisms of $G_{\bar{k}}$ is equal to $G_{\bar{k}}^{\recht{op}}$, which acts on $G_{\bar{k}}$ by multiplication on the right. Let $e \in G_{\bar{k}}$ be the unit element. Then $\recht{Stab}(f) \subset G_{\bar{k}}^{\recht{op}}$ is equal to $\recht{Stab}_{G(\bar{k})}(f(e))^{\recht{op}}$, the stabiliser of $f(e)$ in the $G(\bar{k})$-set $X(\bar{k})$; hence 
$$\recht{dim}(\recht{Aut}(G_{\bar{k}},f)) = \recht{dim}(G) - \recht{dim}(G(\bar{k}) \cdot f(e)).$$
Since dimension is invariant under field extensions, we find that $\recht{dim}(\recht{Aut}(P(C))) = \recht{dim}(G) - \recht{dim}(C)$ for all $C \in {}_G X$.
\end{rem}

To prove Theorem \ref{zetaquotient} we need a few intermediate results. Let $q'$ be a power of $q$, and let $\Gamma$ be a finite étale group scheme over $\mathbb{F}_{q'}$, considered as an abstract group together with an action of $\recht{Gal}(\bar{\mathbb{F}}_q/\mathbb{F}_{q'})$. Let $\Gamma$ act on itself from the left via
$$\gamma \cdot \gamma' = \gamma\gamma' \recht{Fr}_{q'}(\gamma)^{-1}.$$
Let $\recht{Conj}_{q'}(\Gamma)$ be the set of orbits in $\Gamma$ under this action.

\begin{lem}
Let $k = \mathbb{F}_q$ be a finite field, let $G$ be an algebraic group over $k$, and let $X$ be a scheme over $k$ with a left action of $G$. Let $q'$ be a power of $q$. Let $C \in {}_G X(\mathbb{F}_{q'})$, and let $(T,f) = P(C)$ be the isomorphism class in $[X \backslash G](\bar{k})$ corresponding to $C$. Let $(\tilde{T},\tilde{f})$ be a model of $(T,f)$ over $\mathbb{F}_{q'}$, and let $\hat{M}$ be the $\mathbb{F}_{q'}$-group scheme $\recht{Aut}(\tilde{T},\tilde{f})^{\recht{red}}$. Regard the finite étale group scheme $\pi_0(\hat{M})$ as an abstract group with an action of $\recht{Gal}(\bar{\mathbb{F}}_q/\mathbb{F}_{q'})$.
\begin{enumerate}
\item There is a canonical bijective correspondence
$$\{(Y,g) \in [[X \backslash G](\mathbb{F}_{q'})] : (Y,g)_{\bar{k}} \cong (T,f)\} \stackrel{\sim}{\leftrightarrow} \recht{Conj}_{q'}(\pi_0(\hat{M})).$$
\item Suppose the identity component $M$ of $\hat{M}$ is unipotent. In the correspondence above, let $(Y,g)$ represent the isomorphism class in $[X \backslash G](\mathbb{F}_{q'})$ corresponding to $c \in \recht{Conj}_{q'}(\pi_0(\hat{M}))$. Then
$$\# \recht{Aut}(Y,g)(\mathbb{F}_{q'}) = q'^{\recht{dim}(\hat{M})} \cdot \frac{\#\pi_0(\hat{M})}{\# c}.$$
\end{enumerate}
\end{lem}

\begin{proof} { \ }
\begin{enumerate}
\item By \cite[Théorème 2.5.1]{giraud1966} the forms of $(T,f)$ over $\mathbb{F}_{q'}$ are classified by $\recht{H}^1(\mathbb{F}_{q'},\recht{Aut}(\tilde{T},\tilde{f}))$. Since $\recht{H}^1(\mathbb{F}_{q'},\recht{Aut}(\tilde{T},\tilde{f}))$ only depends on the abstract group $\recht{Aut}(\tilde{T},\tilde{f})(\bar{\mathbb{F}}_q)$ together with its Galois action, this is equal to $\recht{H}^1(\mathbb{F}_{q'},\hat{M})$. We have an exact sequence
$$\recht{H}^1(\mathbb{F}_{q'},M)  \rightarrow \recht{H}^1(\mathbb{F}_{q'},\hat{M}) \rightarrow \recht{H}^1(\mathbb{F}_{q'},\pi_0(\hat{M})).$$
By Lang's theorem, the first term is trivial, and by \cite[Corollary III.2.4.2]{serre1997} the last map is surjective. We conclude that $\recht{H}^1(\mathbb{F}_{q'},\hat{M}) = \recht{H}^1(\mathbb{F}_{q'},\pi_0(\hat{M}))$. By Lemma \ref{cocykel} the set $\recht{Z}^1(\mathbb{F}_{q'},\pi_0(\hat{M}))$ corresponds bijectively to $\pi_0(\hat{M})$, by sending a cocycle $z$ to $z(\recht{Fr}_{q'})$. A simple calculation shows that two cocycles $z,z'\in \recht{Z}^1(\mathbb{F}_{q'},\pi_0(\hat{M}))$ give rise to the same class in $\recht{H}^1(\mathbb{F}_{q'},\pi_0(\hat{M}))$ if and only if $z(\recht{Fr}_{q'})$ and $z'(\recht{Fr}_{q'})$ are the same in $\recht{Conj}_{q'}(\pi_0(\hat{M}))$.
\item Let $\hat{M}_c$, with identity component $M_c$, be the reduced automorphism group scheme of $(Y,g)$. We get a short exact sequence
$$1 \rightarrow M_c(\mathbb{F}_{q'}) \rightarrow \hat{M}_c(\mathbb{F}_{q'}) \rightarrow \pi_0(\hat{M}_c)(\mathbb{F}_{q'}) \rightarrow \recht{H}^1(\mathbb{F}_{q'},M_c).$$
Again by Lang's theorem the last term is trivial, hence $\#\hat{M}_c(\mathbb{F}_{q'}) =  \#\hat{M}^0_c(\mathbb{F}_{q'}) \cdot \# \pi_0(\hat{M}_c)(\mathbb{F}_{q'})$. We know that the identity component $M_c$ of $\hat{M}_c$ is a form of $M$ over $\mathbb{F}_{q'}$. By assumption $M$ is unipotent, hence so is $M_c$. This implies that $M_c \cong \mathbb{A}^{\recht{dim}(\hat{M})}$ as varieties over $\mathbb{F}_{q'}$ (see \cite[Theorem 5]{rosenlicht1963}), so $\#M_c(\mathbb{F}_{q'}) = q'^{\recht{dim}(M)}$. Let $y \in c \subset \pi_0(\hat{M})$. Then as an abstract group $\pi_0(M_c) \cong \pi_0(M)$, but the Galois action is given by $\tilde{\sigma}(x) = y \recht{Fr}_{q'}(x) y^{-1}$, for all $x \in \pi_0(\hat{M})$, where $\recht{Fr}_{q'}$ is the regular Galois action on $\pi_0(\hat{M})$. Hence
\begin{align*}
\# \pi_0(\hat{M}_c) &= \#\{x \in \pi_0(\hat{M}): \tilde{\sigma}(x) = x\}\\
&= \#\{x \in \pi_0(\hat{M}): x^{-1} y \recht{Fr}_{q'}(x) = y\} \\
&= \# \recht{Stab}_{\pi_0(\hat{M})}(y) \\
&= \frac{\# \pi_0(\hat{M})}{\# c},
\end{align*}
from which we see that $\#\hat{M}_c(\mathbb{F}_{q'}) = q'^{\recht{dim}(\hat{M})} \cdot \frac{\#\pi_0(\hat{M})}{\# c}.$ \qedhere
\end{enumerate}
\end{proof}

\begin{proof}[Proof of Theorem \ref{zetaquotient}]
To prove this theorem, we need to prove that
$$\sum_{(T,f) \in [[G \backslash X](\mathbb{F}_{q^v})]} \frac{1}{\#\recht{Aut}(T,f)(\mathbb{F}_{q^v})} = \sum_{C \in {}_G X(\mathbb{F}_{q^v})} q^{-\recht{dim}(\recht{Aut}(P(C)))\cdot v}.$$
To prove this, choose for every $v \geq 1$ and every $C \in {}_G X(\mathbb{F}_{q^v})$ a model for $P(C)$ over $\mathbb{F}_{q^v}$, and let $\hat{M}_{C}$ be the automorphism group scheme of this model. Then
\begin{align*}
\sum_{(T,f) \in [[G \backslash X](\mathbb{F}_{q^v})]} \frac{1}{\#\recht{Aut}(T,f)(\mathbb{F}_{q^v})} &=  \sum_{C \in {}_G X} \left( \sum_{\substack{(T,f) \in [[G \backslash X](\mathbb{F}_{q^v})] : \\ (T,f)_{\bar{k}} \cong P(C)}} \frac{1}{\#\recht{Aut}(T,f)(\mathbb{F}_{q^v})} \right)\\
&= \sum_{C \in {}_G X(\mathbb{F}_{q^v})} \left( \sum_{c \in \recht{Conj}_{q^v}(\pi_0(\hat{M}_C))} q^{-v \cdot \recht{dim}(M_{\bar{x}})} \cdot \frac{\# c}{\# \pi_0(\hat{M}_{\bar{x}})} \right)\\
&= \sum_{C \in {}_G X(\mathbb{F}_{q^v})}  q^{-v \cdot \recht{dim}(\hat{M}_{\bar{x}})},
\end{align*}
as was to be shown.
\end{proof}

\subsection{Quotients of affine space}

We now apply the knowledge gained in the previous section to the case where $X$ is an affine space over $k$. This example will be of use when discussing stacks of truncated Barsotti--Tate groups in section \ref{sectieBT}.

\begin{thm} \label{affienquotient}
Let $k = \mathbb{F}_q$ be a finite field. Let $G'$ be a connected algebraic group over $k$, and let $G \subset G'$ be a linear algebraic subgroup whose identity component is unipotent. Let $X$ be an affine space over $k$ with an action of $G'$. Then
$$Z([G \backslash X],t) = \frac{1}{1-q^{\recht{dim}(X)-\recht{dim}(G)}t}.$$
\end{thm}

For the proof we need some intermediate results.

\begin{lem} \label{eindigebanen}
Let $G'$ be a connected algebraic group over $k = \mathbb{F}_q$, acting on $\mathbb{A}^n_k$. Let $F \subset G'$ be a finite subgroup. Then for every $v \geq 1$ one has $\#{}_F\mathbb{A}^n(\mathbb{F}_{q^v}) = q^{nv}$ (using the notation of Remark \ref{fodopmerking}).
\end{lem}

\begin{proof}
For convenience, we identify $F$ with its group of $\bar{k}$-points; the difference between $F$ and $F^{\recht{red}}$ is not relevant for this lemma. Let $v$ be a positive integer, and let $\sigma$ be the action of $\recht{Fr}_{q^v}$ on $\mathbb{A}^n$. Let $x \in \mathbb{A}^n(\bar{k})$; then $F \cdot x$ is defined over $\mathbb{F}_{q^v}$ if and only if $\sigma(x) = f \cdot x$ for some $f \in F$. As such we find
$$\bigcup_{C \in {}_G \mathbb{A}^n(\mathbb{F}_{q^v})} C = \bigcup_{f \in F} \{x \in \mathbb{A}^n(\bar{k}): \sigma(x) = fx\}.$$
Let $C$ be an element of ${}_G \mathbb{A}^n(\mathbb{F}_{q^v})$, and let $c$ be an element of $C$. The number of sets of the form $\{x \in \mathbb{A}^n(\bar{k}): \sigma(x) = fx\}$ of which $c$ is an element, is equal to $\#\recht{Stab}_{F}(c) = \frac{\# F}{\# C}$. Hence in the sum
$$\sum_{f \in F} \#\{x \in \mathbb{A}^n(\bar{k}): \sigma(x) = fx\}$$
every orbit $C$ is counted exactly $\# F$ times. On the other hand, $\#\{x \in \mathbb{A}^n(\bar{k}): \sigma(x) = fx\}$ is equal to the number of $\mathbb{F}_{q^v}$-points of the twist of $\mathbb{A}^n$ over $\mathbb{F}_{q^v}$ given by the cocycle in $\recht{Z}^1(\mathbb{F}_{q^v},\recht{Aut}(\mathbb{A}^n))$ that sends $\recht{Fr}_{q^v}$ to $f^{-1}$. As a Galois cohomology class, this cocycle lies within the image of the composite morphism
$$\recht{H}^1(\mathbb{F}_{q^v},F) \rightarrow \recht{H}^1(\mathbb{F}_{q^v},G') \rightarrow \recht{H}^1(\mathbb{F}_{q^v},\recht{Aut}(\mathbb{A}^n)).$$
By Lang's theorem, the middle term is trivial, hence this is a trivial twist of $\mathbb{A}^n(\bar{k})$. This implies that $\#\{x \in \mathbb{A}^n(\bar{k}): \sigma(x) = fx\} = q^{nv}$, so
$$\#{}_F\mathbb{A}^n(\mathbb{F}_{q^v}) = \frac{1}{\# F} \sum_{f \in F} \#\{x \in \mathbb{A}^n(\bar{k}): \sigma(x) = fx\} = q^{nv}. \qedhere$$
\end{proof}

\begin{rem}
The condition that the action of $F$ on $\mathbb{A}^n_k$ factors through a connected algebraic group is somewhat artificial, but it ensures that the resulting twists of affine space are trivial. To my knowledge, it is an open problem whether there exist nontrivial twists of affine space over a finite field.
\end{rem}

\begin{lem} \label{unipotentebanen}
Let $k = \mathbb{F}_q$ be a finite field. Let $G'$ be a connected algebraic group over $k$, and let $G \subset G'$ be a linear algebraic subgroup whose identity component is unipotent. Let $X$ be an affine space over $k$ with an action of $G$. Then
$$\sum_{C \in {}_G X(\mathbb{F}_{q^v})} q^{\recht{dim}(C)\cdot v} = q^{\recht{dim}(X) \cdot v}.$$
\end{lem}

\begin{proof}
Let $G^0$ be the identity component of $G$. By \cite[Theorem 1.1]{brion2015} there exists a finite subgroup $F \subset G$ such that $G = G^0 \cdot F$; as before, we identify $F$ with its set of $\bar{k}$-points. Let $v$ be a positive integer, and let $\sigma$ denote the action of $\recht{Fr}_{q^v}$ on $X$. By Lemma \ref{eindigebanen} there are precisely $q^{\recht{dim}(X) \cdot v}$ orbits in $X$ under the action of $F$, that are defined over $\mathbb{F}_{q^v}$. Hence it is enough to show that every $C \in {}_GX(\mathbb{F}_{q^v})$ contains exactly $q^{\recht{dim}(C)}$ of these orbits. Let $C \in {}_GX(\mathbb{F}_{q^v})$, and let $C^0$ be a connected component of $C$; this component is a homogeneous $G^0$-space. Let $f \in F(\bar{k})$ be such that $\sigma(C^0) = fC^0$. Let $x \in X(\bar{k})$ such that $x \in C^0$ and such that $F \cdot x$ is defined over $\mathbb{F}_{q^v}$. Let $F' = \recht{Stab}_{F}(C^0)$; then there is a $f' \in F'$ such that $\sigma(x) = ff'x$. Analogous to the proof of Lemma \ref{eindigebanen} we can consider the sum
$$\sum_{f' \in F'} \#\{x \in C^0(\bar{k}): \sigma(x) = ff'x\}.$$
As in the proof of Lemma \ref{eindigebanen}, each $F'$-orbit in $C^0(\bar{k})$, whose $F$-orbit in $C$ defined over $\mathbb{F}_{q^v}$, is counted exactly $\#F'$ times in this sum. Since the correspondence between $F$-orbits in $C$ defined over $\mathbb{F}_{q^v}$ and $F'$-orbits in $C^0$ whose associated $F$-orbit is defined over $\mathbb{F}_{q^v}$ is bijective, we now only have to prove that 
$$\sum_{f' \in F'} \#\{x \in C^0(\bar{k}): \sigma(x) = ff'x\} = \# F \cdot q^{\recht{dim}(C) \cdot v}.$$
To prove this, we define, for a fixed $f' \in F'$, a twisted Galois action on $C^0(\bar{k})$ by $\tilde{\sigma}(x) := (ff')^{-1}\sigma(x) \in C^0(\bar{k})$; this defines a variety $\tilde{C}^0$ over $\mathbb{F}_{q^v}$. This has a transitive action by a twist $\tilde{G}^0$ of $G^0$ over $\mathbb{F}_{q^v}$, where the Galois action on $\tilde{G}^0(\bar{k}) = G^0(\bar{k})$ is defined as $\tilde{\sigma}(g) = (ff')^{-1}\sigma(g)ff'$. Then $\tilde{G}^0$ is a connected unipotent group over $\mathbb{F}_{q^v}$, and as such its homogenous space $\tilde{C}^0$ is isomorphic (as a variety over $\mathbb{F}_{q^v}$) to $\mathbb{A}^{\recht{dim}(C)}_{\mathbb{F}_{q^v}}$ (see \cite[Theorem 5]{rosenlicht1963}). As such we find $\#\{x \in C^0(\bar{k}): \sigma(x) = ff'x\} =\#\tilde{C}^0(\mathbb{F}_{q^v}) = q^{\recht{dim}(C) \cdot v},$ which proves the lemma.
\end{proof}

\begin{proof}[Proof of Theorem \ref{affienquotient}] By Theorem \ref{zetaquotient}, Remark \ref{zetaopm} and Lemma \ref{unipotentebanen} we have
\begin{align*}
Z([G \backslash X],t) &= \recht{exp}\left(\sum_{v \geq 1} \frac{t^v}{v} \sum_{C \in {}_G X(\mathbb{F}_{q^v})} q^{-\recht{dim}(\recht{Aut}(P(C))) \cdot v}\right) \\
&= \recht{exp}\left(\sum_{v \geq 1} \frac{t^v}{v} \sum_{C \in {}_G X(\mathbb{F}_{q^v})} q^{(\recht{dim}(C)-\recht{dim}(G)) \cdot v}\right) \\
&= \recht{exp}\left(\sum_{v \geq 1} \frac{t^v}{v}  q^{(\recht{dim}(X)-\recht{dim}(G)) \cdot v}\right) \\
&= \frac{1}{1-q^{\recht{dim}(X)-\recht{dim}(G)}t}. \qedhere
\end{align*}

\end{proof}

\section{Weyl groups and Levi decompositions} \label{weylgroepen}

In this section we briefly review some relevant facts about Weyl groups and Levi decompositions, in particular those of nonconnected reductive groups.

\subsection{The Weyl group of a connected reductive group}

Let $G$ be a connected reductive algebraic group over a field $k$. For any pair $(T,B)$ of a Borel subgroup $B \subset G_{\bar{k}}$ and a maximal torus $T \subset B$, let $\Phi_{(T,B)}$ be the based root system of $G$ with respect to $(T,B)$, and let $W_{T,B}$ be the Weyl group of this based root system, i.e. the Coxeter group generated by the set $S_{T,B}$ of simple reflections. As an abstract group $W_{T,B}$ is isomorphic to $\recht{Norm}_{G(\bar{k})}(T(\bar{k}))/T(\bar{k})$. If $(T',B')$ is another choice of a Borel subgroup and a maximal torus, then there exists a $g \in G(\bar{k})$ such that $(T',B') = (gTg^{-1},gBg^{-1})$. Furthermore, such a $g$ is unique up to right multiplication by $T(\bar{k})$, which gives us a unique isomorphism $\Phi_{T,B} \stackrel{\sim}{\rightarrow} \Phi_{T',B'}$. As such, we can simply talk about \emph{the} based root system $\Phi$ of $G$, with corresponding Coxeter system $(W,S)$. By these canonical identifications $\Phi$, $W$ and $S$ come with an action of $\recht{Gal}(\bar{k}/k)$.\\

The set of parabolic subgroups of $G_{\bar{k}}$ containing $B$ is classified by the power set of $S$, by associating to $I \subset S$ the parabolic subgroup $P = L \cdot B$, where $L$ is the reductive group with maximal torus $T$ whose root system is $\Phi_{I}$, the root subsystem of $\Phi$ generated by the roots whose associated reflections lie in $I$. We call $I$ the \emph{type} of $P$. Let $U = \recht{R}_{\recht{u}}P$ be the unipotent radical of $P$; then $P = L \ltimes U$ is the Levi decomposition of $P$ with respect to $T$ (see subsection \ref{sectielevi}). For every subset $I \subset S$, let $W_I$ be the subgroup of $W$ generated by $I$; it is the Weyl group of the root system $\Phi_I$, with $I$ as its set of simple reflections. \\

For $w \in W$, define the \emph{length} $\ell(w)$ of $w$ to be the minimal integer such that there exist $s_1,s_2,...,s_{\ell(w)} \in S$ such that $w = s_1s_2\cdots s_{\ell(w)}$. Since $\recht{Gal}(\bar{k}/k)$ acts on $W$ by permuting $S$, the length is Galois invariant. Let $I,J \subset S$; then every (left, double, right) coset of the form $W_Iw, W_IwW_J$ or $wW_J$ has a unique element of minimal length, and we denote the subsets of $W$ of elements of minimal length in their (left, double, right) cosets by ${}^IW, {}^IW^J$, and $W^J$.

\begin{prop} \label{deng418}
(See \cite[Proposition 4.18]{deng2008}) Let $I,J \subset S$. Let $x \in {}^IW^J$, and set $I_x = J \cap x^{-1}Ix \subset W$. Then for every $w \in W_IxW_J$ there exist unique $w_I \in W_I$, $w_J \in {}^{I_x}W_J$ such that $w = w_Ixw_J$. Furthermore $\ell(w) = \ell(w_I)+\ell(x)+\ell(w_J)$. \qed
\end{prop}

\begin{lem} \label{pwz28}
(See \cite[Proposition 2.8]{pinkwedhornziegler2011}) Let $I,J \subset S$. Every element $w \in {}^IW$ can uniquely be written as $xw_J$ for some $x \in {}^IW^J$ and $w_J \in {}^{I_x}W_J$. \qed
\end{lem}

\begin{lem} \label{pwz213}
(See \cite[Lemma 2.13]{pinkwedhornziegler2011}) Let $I,J \subset S$. Let $w \in {}^IW$ and write $w = xw_J$ with $x \in {}^IW^J$, $w_J \in W_J$. Then 

\[
\pushQED{\qed}
\ell(w) = \#\{\alpha \in \Phi^+ \backslash \Phi_J : w\alpha \in \Phi^- \backslash \Phi_I\}. \qedhere
\popQED
\]

\end{lem}

\subsection{The Weyl group of a nonconnected reductive group} \label{weylgroepnietsamenhangend}

Now let us drop the assumption that our group is connected.  Let $\hat{G}$ be a reductive algebraic group and write $G$ for its connected component. Let $B$ be a Borel subgroup of $G_{\bar{k}}$, and let $T$ be a maximal torus of $B$. Define the following groups:

\begin{align*}
W &= \recht{Norm}_{G(\bar{k})}(T)/T(\bar{k}); \\
\hat{W} &= \recht{Norm}_{\hat{G}(\bar{k})}(T)/T(\bar{k}); \\
\Omega &= (\recht{Norm}_{\hat{G}(\bar{k})}(T) \cap \recht{Norm}_{\hat{G}(\bar{k})}(B))/T(\bar{k}).
\end{align*}

\begin{lem} \label{Weylgroep} { \ } 
\begin{enumerate}
\item One has $\hat{W}  = W \rtimes \Omega$.
\item The composite map $\Omega \hookrightarrow G(\bar{k})/T(\bar{k}) \twoheadrightarrow \pi_0(G(\bar{k}))$ is an isomorphism of groups.
\end{enumerate}
\end{lem}

\begin{proof} { \ }
\begin{enumerate}
\item First note that $W$ is a normal subgroup of $\hat{W}$, since it consists of the elements of $\hat{W}$ that have a representative in $G(\bar{k})$, and $G$ is a normal subgroup of $\hat{G}$. Furthermore, $\hat{W}$ acts on the set $X$ of Borel subgroups of $G_{\bar{k}}$ containing $T$. The stabiliser of $B$ under this action is $\Omega$, whereas $W$ acts simply transitively on $X$; hence $\Omega \cap W = 1$ and $W\Omega = \hat{W}$, and together this proves $\hat{W} = W \rtimes \Omega$.
\item By the previous point, we see that 
$$\Omega \cong \hat{W}/W \cong \recht{Norm}_{\hat{G}(\bar{k})}(T)/\recht{Norm}_{G(\bar{k})}(T),$$
so it is enough to show that every connected component of $\hat{G}(\bar{k})$ has an element that normalises $T$. Let $x \in \hat{G}(\bar{k})$; then $xTx^{-1}$ is another maximal torus of $G_{\bar{k}}$, so there exists a $g \in G(\bar{k})$ such that $xTx^{-1} = gTg^{-1}$. From this we find that $T = (g^{-1}x)T(g^{-1}x)^{-1}$, and $g^{-1}x$ is in the same connected component as $x$. \qedhere
\end{enumerate}
\end{proof}

We call $\hat{W}$ the Weyl group of $\hat{G}$ with respect to $(T,B)$. Again, choosing a different $(T,B)$ leads to a canonical isomorphism, so we may as well talk about \emph{the} Weyl group of $G$. The two statements of Lemma \ref{Weylgroep} are then to be understood as isomorphisms of groups with an action of $\recht{Gal}(\bar{k}/k)$. Note that we can regard $W$ as the Weyl group of the connected reductive group $G$; as such we can apply the results of the previous subsection to it. Let $S \subset W$ be the generating set of simple reflections.\\

Now let us define an extension of the length function to a suitable subset of $\hat{W}$. First, let $I$ and $J$ be subsets of the set $S$ of simple reflections in $W$, and consider the set ${}^I\hat{W} := {}^IW\Omega$. Define a subset ${}^I\hat{W}^J$ of ${}^I\hat{W}$ as follows: every element $\hat{w} \in {}^I\hat{W}$ can uniquely be written as $\hat{w} = w'\omega$, with $w' \in {}^IW$ and $\omega \in \Omega$. We rewrite this as $\hat{w} = \omega w''$, with $w'' = \omega^{-1}w'\omega \in {}^{\omega^{-1}I\omega}W$; then per definition $\hat{w} \in  {}^I\hat{W}^J$ if and only if $w'' \in {}^{\omega^{-1}I\omega}W^J$. Note that ${}^IW^J$ is contained in ${}^I\hat{W}^J$.\\

Now let $\hat{w} \in {}^I\hat{W}$; write $\hat{w} = \omega w''$. Since $w''$ is an element of ${}^{\omega^{-1}I\omega}W$, we can uniquely write $w'' = yw_J$ by Lemma \ref{pwz28}, with $y \in {}^{\omega^{-1}I\omega}W^J$ and $w_J \in {}^{I_{\omega y}}W_J$. Then define
\begin{equation} \label{vergelijking}
\ell_{I,J}(\hat{w}) := \#\{\alpha \in \Phi^+\backslash \Phi_J: \omega y \alpha \in \Phi^-\backslash \Phi_I\} + \ell(w_J).
\end{equation}
\begin{rem} \label{opmerking} { \ }
\begin{enumerate}
\item By Proposition \ref{deng418} and Lemma \ref{pwz213} the map $\ell_{I,J}\colon{}^I\hat{W} \rightarrow \mathbb{Z}_{\geq 0}$ extends $\ell\colon {}^IW \rightarrow \mathbb{Z}_{\geq 0}$.
\item Analogously to Proposition \ref{deng418} we see that every $w \in {}^I\hat{W}$ can be uniquely written as $xw_J$ with $x \in {}^I\hat{W}^J$, $w_J \in {}^{I_{x}} W_J$, and $\ell_{I,J}(w) = \ell_{I,J}(x)+\ell(w_J)$.
\item A straightforward calculation shows that if $\gamma$ is an element of $\recht{Gal}(\bar{k}/k)$ we get $\ell_{{}^\gamma I,{}^\gamma J}({}^\gamma w) = \ell_{I,J}(w)$. In particular, if $I$ and $J$ are fixed under the action of $\recht{Gal}(\bar{k}/k)$, the map $\ell_{I,J}\colon {}^I\hat{W} \rightarrow \mathbb{Z}_{\geq 0}$ is Galois-invariant.
\item In general $\ell_{I,J}$ depends on $J$. It also depends on $I$, in the sense that if $I, I' \subset S$, then $\ell_{I,J}(w)$ and $\ell_{I',J}(w)$ for $w \in {}^I\hat{W} \cap {}^{I'}\hat{W} = {}^{I \cap I'}\hat{W}$ need not coincide. As an example, consider over any field the group $G = \recht{SL}_2$. Let $\Omega = \langle \omega \rangle$ be cyclic of order 2, and let $\hat{G} = G \rtimes \Omega$ be the extension given by $\omega g \omega^{-1} = g^{\recht{T},-1}$. Then $\omega$ acts as $-1$ on the root system, and $S$ has only one element. Then a straightforward calculation shows $\ell_{\varnothing,\varnothing}(\omega) = 1$, whereas $\ell_{\varnothing,S}(\omega) = \ell_{S,S}(\omega) = \ell_{S,\varnothing}(\omega) = 0$.
\end{enumerate}
\end{rem}

\subsection{Levi decomposition of nonconnected groups} \label{sectielevi}

Let $P$ be a connected smooth linear algebraic group over a field $k$. A Levi subgroup of $P$ is the image of a section of the map $P \twoheadrightarrow P/\recht{R}_{\recht{u}}P$, i.e. a subgroup $L \subset P$ such that $P = L \ltimes \recht{R}_{\recht{u}}P$. In characteristic $p$, such a Levi subgroup need not always exist, nor need it be unique. However, if $P$ is a parabolic subgroup of a connected reductive algebraic group, then for every maximal torus $T \subset P$ there exists a unique Levi subgroup of $P$ containing $T$ (see \cite[Proposition 1.17]{dignemichel1991}). The following Proposition generalises this result to the non-connected case.

\begin{prop} \label{levi}
Let $\hat{G}$ be a reductive group over a field $k$, and let $\hat{P}$ be a subgroup of $\hat{G}$ whose identity component $P$ is a parabolic subgroup of $G$. Let $T$ be a maximal torus of $P$. Then there exists a unique Levi subgroup of $\hat{P}$ containing $T$, i.e. a subgroup $\hat{L} \subset \hat{P}$ such that $\hat{P} = \hat{L} \ltimes \recht{R}_{\recht{u}}P$.
\end{prop}

\begin{proof}
Let $L$ be the Levi subgroup of $P$ containing $T$. Then any $\hat{L}$ satisfying the conditions of the Proposition necessarily has $L$ has its identity component, hence $\hat{L} \subset \recht{Norm}_{\hat{P}}(L)$. On the other hand we know that $\recht{Norm}_{P}(L) = L$, so the only possibility is $\hat{L} = \recht{Norm}_{\hat{P}}(L)$, and we have to check that $\pi_0(\recht{Norm}_{\hat{P}}(L)) = \pi_0(\hat{P})$, i.e. that every connected component in $\hat{P}_{\bar{k}}$ has an element normalising $L$. Let $x \in \hat{P}(\bar{k})$. Then $xTx^{-1}$ is another maximal torus of $P_{\bar{k}}$, so there exists a $p \in P(\bar{k})$ such that $xTx^{-1} = pTp^{-1}$. Then $p^{-1}x$ is in the same connected component as $x$, and $(p^{-1}x)T(p^{-1}x)^{-1} = T$. Since $L$ is the unique Levi subgroup of $P$ containing $T$, and $(p^{-1}x)L(p^{-1}x)^{-1}$ is another Levi subgroup of $P$, we see that $p^{-1}x$ normalises $L$, which completes the proof.
\end{proof}

\section{$\hat{G}$-zips} \label{sectiegzips}

In this section we give the definition of $G$-zips from \cite{pinkwedhornziegler2015}. We also apply the discussion on Weyl groups from section \ref{weylgroepen} to this situation.\\

As before, let $p$ be a prime number. Let $q_0$ be a power of $p$. Let $\hat{G}$ be a reductive group over $\mathbb{F}_{q_0}$, and write $G$ for its identity component. Let $q$ be a power of $q_0$,  and let $\chi\colon \mathbb{G}_{\recht{m},\mathbb{F}_{q}} \rightarrow G_{\mathbb{F}_{q}}$ be a cocharacter of $G_{\mathbb{F}_{q}}$. Let $L = \recht{Cent}_{G_{\mathbb{F}_{q}}}(\chi)$, and let $U_+ \subset G_{\mathbb{F}_q}$ be the unipotent subgroup defined by the property that $\recht{Lie}(U_+) \subset \recht{Lie}(G_{\mathbb{F}_q})$ is the direct sum of the weight spaces of positive weight; define $U_-$ similarly. Note that $L$ is connected (see \cite[Proposition 0.34]{dignemichel1991}). This defines parabolic subgroups $P_{\pm} = L \ltimes U_{\pm}$ of $G_{\mathbb{F}_q}$. Now take an $\mathbb{F}_q$-subgroup scheme $\Theta$ of $\pi_0(\recht{Cent}_{\hat{G}_{\mathbb{F}_q}}(\chi))$, and let $\hat{L}$ be the inverse image of $\Theta$ under the canonical map $\recht{Cent}_{\hat{G}_{\mathbb{F}_q}}(\chi) \rightarrow \pi_0(\recht{Cent}_{\hat{G}_{\mathbb{F}_q}}(\chi))$; then $\hat{L}$ has $L$ as its identity component and $\pi_0(\hat{L}) = \Theta$. We may regard $\Theta$ as a subgroup of $\pi_0(\hat{G})$ via the inclusion
$$\pi_0(\recht{Cent}_{\hat{G}_{\mathbb{F}_q}}(\chi)) = \recht{Cent}_{\hat{G}_{\mathbb{F}_q}}(\chi)/L \hookrightarrow \pi_0(\hat{G}_{\mathbb{F}_q}).$$
We may then define the algebraic subgroups $\hat{P}_{\pm} := \hat{L} \ltimes U_{\pm}$ of $G_{\mathbb{F}_q}$, whose identity components $P_{\pm}$ are equal to $L \ltimes U_{\pm}$. Let $\sigma:G \rightarrow G$ be the relative Frobenius isogeny over $\mathbb{F}_{q_0}$; We may then regard $\sigma(\hat{P}_{\pm})$, $\sigma(\hat{L})$, etc. as subgroups of $G$; they correspond to the cocharacter $\sigma(\chi)$ of $G_{\mathbb{F}_q}$ and the subgroup $\sigma(\Theta)$ of $\pi_0(\hat{G})$.

\begin{defn}
Let $A$ be an algebraic group over a field $k$, and let $B$ be a subgroup of $A$. Let $T$ be an $A$-torsor over some $k$-scheme $S$. A \emph{$B$-subtorsor of $T$} is an $S$-subscheme $Y$ of $T$, together with an action of $B_S$, such that $Y$ is a $B$-torsor over $S$ and such that the inclusion map $Y \hookrightarrow T$ is equivariant under the action of $B_S$.
\end{defn}

\begin{defn}
Let $S$ be a scheme over $\mathbb{F}_q$. A \emph{$\hat{G}$-zip of type $(\chi,\Theta)$ over $S$} is a tuple $\mathcal{Y} = (Y, Y_+, Y_-, \upsilon)$ consisting of:
\begin{itemize}
\setlength\itemsep{0em}
\item A right-$\hat{G}_{\mathbb{F}_q}$-torsor $Y$ over $S$;
\item A right-$\hat{P}_+$-subtorsor $Y_+$ of $Y$;
\item A right-$\sigma(\hat{P}_-)$-subtorsor $Y_-$ of $Y$;
\item An isomorphism $\upsilon\colon \sigma(Y_+)/\sigma(U_+) \rightarrow Y_-/\sigma(U_-)$ of right-$\sigma(\hat{L})$-torsors.
\end{itemize}
\end{defn}

Together with the obvious notions of pullbacks and morphisms we get a fibred category $\hat{G}\textrm{-}\mathsf{Zip}^{\chi,\Theta}_{\mathbb{F}_q}$ over $\mathbb{F}_q$.

\begin{prop} (See \cite[Propositions 3.2 \& 3.11]{pinkwedhornziegler2015})
The fibred category $\hat{G}\textrm{-}\mathsf{Zip}^{\chi,\Theta}_{\mathbb{F}_q}$ is a smooth algebraic stack of finite type over $\mathbb{F}_q$. 
\end{prop}

Now let $q_0,q,\hat{G},\chi,\Theta,\hat{L},U_{\pm}$ and $\hat{P}_{\pm}$ be as above. As in section \ref{weylgroepen} let $\hat{W} = W \rtimes \Omega$ be the Weyl group of $\hat{G}$. Let $I \subset S$ be the type of $P_+$ and let $J$ be the type of $\sigma(P_-)$. If $w_0 \in W$ is the unique longest word, then $J = \sigma(w_0Iw_0^{-1}) = w_0\sigma(I)w_0^{-1}$. Let $w_1 \in {}^JW^{\sigma(I)}$ be the element of minimal length in $W_Jw_0W_{\sigma(I)}$, and let $w_2 = \sigma^{-1}(w_1)$; then we may write this relation as $J = \sigma(w_2Iw_2^{-1}) = w_1\sigma(I)w_1^{-1}$.\\

The group $\Theta$ can be considered as a subgroup of $\Omega \cong \pi_0(\hat{G})$. Let $\hat{\psi}$ be the automorphism of $\hat{W}$ given by $\hat{\psi} = \recht{inn}(w_1) \circ \sigma = \sigma \circ \recht{inn}(w_2)$, and let $\Theta$ act on $\hat{W}$ by $\theta \cdot w = \theta w \hat{\psi}(\theta)^{-1}$.

\begin{lem}
The subset ${}^I\hat{W} \subset \hat{W}$ is invariant under the $\Theta$-action.
\end{lem}

\begin{proof}
Since  $\hat{L}$ normalises the parabolic subgroup $P_+$ of $G_{\mathbb{F}_q}$, the subset $I \subset S$ is stable under the action of $\Theta$ by conjugation; hence for each $\theta \in \Theta$ one has $\theta({}^IW)\theta^{-1} = {}^IW$, so 
\[
\theta({}^I\hat{W})\hat{\psi}(\theta)^{-1} = (\theta({}^IW)\theta^{-1}) \cdot (\theta \Omega \hat{\psi}(\theta)^{-1}) = {}^IW\Omega = {}^I\hat{W}. \qedhere
\]
\end{proof}

Let us write $\Xi^{\chi,\Theta} := \Theta \backslash {}^I\hat{W}$.

\begin{prop} \label{classificatie}
(See \cite[Remark 3.21]{pinkwedhornziegler2015}) There is a natural bijection $\Xi^{\chi,\Theta}  \stackrel{\sim}{\leftrightarrow} [\hat{G}\textrm{-}\mathsf{Zip}^{\chi,\Theta}_{\mathbb{F}_q}(\bar{\mathbb{F}}_q)]$.\qed
\end{prop}

This bijection can be described as follows. Choose a Borel subgroup $B$ of $G_{\bar{\mathbb{F}}_q}$ contained in $\sigma(P_-)$, and let $T$ be a maximal torus of $B$. Let $\gamma \in G(\bar{\mathbb{F}}_q)$ be such that $\sigma(\gamma B \gamma^{-1}) = B$ and $\sigma( \gamma T \gamma^{-1}) = T$. For every $w \in \hat{W} = \recht{Norm}_{\hat{G}(\bar{\mathbb{F}}_q)}(T)/T(\bar{\mathbb{F}}_q)$, choose a lift $\dot{w}$ to $\recht{Norm}_{\hat{G}(\bar{\mathbb{F}}_q)}(T)$, and set $g = \gamma\dot{w}_2$. Then $\xi \in \Xi^{\chi,\Theta}$ corresponds to the $\hat{G}$-zip $\mathcal{Y}_{w} = (\hat{G},\hat{P}_+,g\dot{w}\sigma(\hat{P}_-),g\dot{w} \cdot )$ for any representative $w \in {}^I\hat{W}$ of $\xi$; its isomorphism class does not depend on the choice of the representatives $w$ and $\dot{w}$. Note that this description differs from the one given in \cite[Remark 3.21]{pinkwedhornziegler2015}, as the description given there seems to be wrong. Since there it is assumed that $B \subset P_{-,K}$ rather than that $B \subset P_{-,K}^{(q)}$, the choice of $(B,T,g)$ presented there will not be a frame for the connected zip datum $(G_K,P_{+,K},P_{-,K}^{(q)}, \varphi \colon L_K \rightarrow L^{(q)}_K)$. Also, the choice for $g$ given there needs to be modified to account for the fact that $P_{+,K}$ and $P_{-,K}^{(q)}$ might not have a common maximal torus.\\

The rest of this subsection is dedicated to the extended length functions defined in subsection \ref{weylgroepnietsamenhangend}. We need Lemma \ref{thetainv} in order to show a result on the dimension of the automorphism group of a $\hat{G}$-zip that extends \cite[Proposition 3.34(a)]{pinkwedhornziegler2015} to the nonconnected case (see Proposition \ref{rekenopm}.2).

\begin{lem} \label{thetainv}
The length function $\ell_{I,J}\colon {}^I\hat{W} \rightarrow \mathbb{Z}_{\geq 0}$ is invariant under the semilinear conjugation action of $\Theta$.
\end{lem}

\begin{proof}
Let $w \in {}^I\hat{W}$, let $\theta \in \Theta$, and let $\tilde{w} = \theta w\hat{\psi}(\theta)^{-1}$. Let $w = \omega y w_J$ be the decomposition as in subsection \ref{weylgroepnietsamenhangend}. A straightforward computation shows $\tilde{w} = \tilde{\omega}\tilde{w}' = \tilde{\omega}\tilde{y}\tilde{w}_J$ with
\begin{align*}
\tilde{\omega} &= \theta \omega \hat{\psi}(\theta)^{-1} \in \Omega;\\
\tilde{w} &= \hat{\psi}(\theta) w\hat{\psi}(\theta)^{-1} \in {}^{\tilde{\omega}^{-1}I\tilde{\omega}}W; \\
\tilde{y} &= \hat{\psi}(\theta)y\hat{\psi}(\theta)^{-1} \in {}^{\tilde{\omega}^{-1}I\tilde{\omega}}W^J; \\
\tilde{w}_J &= \hat{\psi}(\theta)w_J\hat{\psi}(\theta)^{-1} \in {}^{I_{\tilde{\omega}\tilde{y}}}W_J,
\end{align*}
since conjugation by $\hat{\psi}(\theta)$ fixes $J$. Furthermore, $\hat{\psi}(\Theta)$ fixes $\Phi_J$ and $\Theta$ fixes $\Phi_I$, and $\Omega$ fixes $\Phi^+$ and $\Phi^-$, hence
\begin{align*}
\ell_{I,J}(\tilde{w}) &= \#\{\alpha \in \Phi^+ \backslash \Phi_J : \tilde{\omega} \tilde{y}\alpha \in \Phi^- \backslash \Phi_I\}+\ell(\tilde{w}_J) \\
&= \#\{\alpha \in \Phi^+ \backslash \Phi_J : \theta\omega y\hat{\psi}(\theta)^{-1}\alpha \in \Phi^- \backslash \Phi_I\}+\ell(\tilde{w}_J) \\
&= \#\{\alpha \in \Phi^+ \backslash \Phi_J : \omega y \alpha \in \Phi^- \backslash \Phi_I\}+\ell(w_J) \\
&= \ell_{I,J}(w). \qedhere
\end{align*}
\end{proof}

Let $\Gamma = \recht{Gal}(\bar{\mathbb{F}}_q/\mathbb{F}_q)$. By Corollary \ref{opmerking}.3 we know that the function $\ell_{I,J}$ is not only invariant under the action of $\Theta$, but also under the action of $\Gamma$. As such, we can also consider $\ell_{I,J}$ as a function $\Xi^{\chi,\Theta}  \rightarrow \mathbb{Z}_{\geq 0}$ or as a function $\Gamma\backslash\Xi^{\chi,\Theta}  \rightarrow \mathbb{Z}_{\geq 0}$. Throughout the rest of this article we will denote elements of $\hat{W}$ (or subsets thereof) by $w$, elements of $\Xi^{\chi,\Theta} $ by $\xi$, and elements of $\Gamma\backslash\Xi^{\chi,\Theta}$ by $\bar{\xi}$.

\begin{exa} \label{voorbeeld1}
Let $p$ be an odd prime, let  $V$ be the $\mathbb{F}_p$-vector space $\mathbb{F}_p^4$, and let $\psi$ be the symmetric nondegenerate bilinear form on $V$ given by the matrix 
$$\left(\begin{array}{cccc} 0 & 0 & 0 & 1 \\ 0 & 0 & 1 & 0 \\ 0 & 1 & 0 & 0 \\ 1 & 0 & 0 & 0\end{array}\right).$$
Let $\hat{G}$ be the algebraic group $\recht{O}(V,\psi)$ over $\mathbb{F}_p$. Then $\hat{G}$ has two connected components. The Weyl group $W$ of its identity component $G = \recht{SO}(V,\psi)$ is of the form $W \cong \{\pm 1\}^2$ (with trivial Galois action), and its root system is of the form $\Psi \cong \{r_1,r_2,-r_1,-r_2\}$, where the $i$-th factor of $W$ acts on $\{r_i,-r_i\}$. The set of generators of $W$ is $S = \{(-1,1),(1,-1)\}$. Furthermore, $\#\Omega = 2$, and the nontrivial element $\sigma$ of $\Omega$ permutes the two factors of $W$ (as well as $e_1$ and $e_2$); hence $\hat{W} \cong \{\pm 1\}^2 \ltimes S_2$.\\

Let $\chi\colon \mathbb{G}_{\recht{m}} \rightarrow G$ be the cocharacter that sends $t$ to $\recht{diag}(t,t,t^{-1},t^{-1})$. Its associated Levi factor $L$ is isomorphic to $\recht{GL}_2$; the isomorphism is given by the injection $\recht{GL}_2 \hookrightarrow \hat{G}$ that sends a $g \in \recht{GL}_2$ to $\recht{diag}(g,g^{-1})$. The associated parabolic subgroup $P_+$ is the product of $L$ with the subgroup $B \subset \hat{G}$ of upper triangular orthogonal matrices. The type of $P_+$ is a singleton subset of $S$; without loss of generality we may choose the isomorphism $W \cong \{\pm 1\}^2$ in such a way that $P^+$ has type $I = \{(-1,1)\}$. Recall that $J$ denotes the type of the parabolic subgroup $\sigma(P_-)$ of $G$. Since $W$ is abelian and has trivial Galois action, the formula $J = w_0\sigma(I)w_0^{-1}$ shows us that $J = I$. Furthermore, since $\recht{Cent}_{\hat{G}}(\chi)$ is connected, the group $\Theta$ has to be trivial.\\

An element of $\hat{W}$ is of the form $(a,b,c)$, with $a,b \in \{\pm 1\}$ and $c \in S_2 = \{1,\sigma\}$; then ${}^I\hat{W}$ is the subset of $\hat{W}$ consisting of elements for which $a=1$. Also, note that $\Phi^+\setminus \Phi_J = \{e_2\}$, $\Phi^- \setminus \Phi_I = \{-e_2\}$, so to calculate the length function $\ell_{I,J}$ as in (\ref{vergelijking})  we only need to determine $\ell(w_J)$ and  whether $\omega y$ sends $e_2$ to $-e_2$ or not. If we use the terminology $\omega$, $w''$, $y$, $w_J$ from section \ref{weylgroepnietsamenhangend}, we get the following results:
\begin{center}
\begin{tabular}{c|cccc|cc|c}
$\hat{w}$ & $\omega$ & $w''$ & $y$ & $w_J$ & $\omega y e_2 = -e_2$? & $\ell(w_J)$ & $\ell_{I,J}(\hat{w})$ \\
\hline
$(1,1,1)$&$(1,1,1)$&$(1,1,1)$&$(1,1,1)$&$(1,1,1)$& no & $0$ & $0$\\
$(1,-1,1)$&$(1,1,1)$&$(1,-1,1)$&$(1,-1,1)$&$(1,1,1)$& yes & $0$ & $1$\\
$(1,1,\sigma)$&$(1,1,\sigma)$&$(1,1,1)$&$(1,1,1)$&$(1,1,1)$& no & $0$ & $0$\\
$(1,-1,\sigma)$&$(1,1,\sigma)$&$(-1,1,1)$&$(1,1,1)$&$(-1,1,1)$& no & $1$ & $1$\\
\end{tabular}
\end{center}
\end{exa}

\section{Algebraic zip data} \label{algebraiczipdata}

In this section we discuss algebraic zip data, which are needed to prove statements about the automorphism group of a $\hat{G}$-zip. This section copies a lot from sections 3--8 of \cite{pinkwedhornziegler2011}, except that there the reductive group is assumed to be to be connected. A lot carries over essentially unchanged; in particular, if we cite a result from \cite{pinkwedhornziegler2011} without comment, we mean that the same proof holds for the nonconnected case. Throughout this section, we will be working over an algebraically closed field $k$ of characteristic $p$, and for simplicity, we will identify algebraic groups with their set of $k$-points (which means that we take all groups to be reduced). Furthermore, if $\hat{A}$ is an algebraic group, then we denote its identity component by $A$. Finally for an algebraic group $\hat{G}$ we denote by $\pi_{\hat{G}}$ the quotient map $\pi_{\hat{G}} \colon \hat{G} \rightarrow \hat{G}/\recht{R}_{\recht{u}}G$.

\begin{defn}
An \emph{algebraic zip datum} over $k$ is a quadruple $\hat{\mathcal{Z}} = (\hat{G},\hat{P},\hat{Q},\varphi)$ consisting of:
\begin{itemize}
\setlength\itemsep{0em}
\item a reductive group $\hat{G}$ over $k$;
\item two subgroups $\hat{P}, \hat{Q} \subset \hat{G}$ such that $P$ and $Q$ are parabolic subgroups of $G$;
\item an isogeny $\varphi\colon \hat{P}/\recht{R}_{\recht{u}}P \rightarrow \hat{Q}/\recht{R}_{\recht{u}}Q$, i.e. a morphism of algebraic groups with finite kernel.
\end{itemize}
\end{defn}

For an algebraic zip datum $\hat{\mathcal{Z}}$, we define its zip group $E_{\hat{\mathcal{Z}}}$ to be
$$E_{\hat{\mathcal{Z}}} = \{(p,q) \in \hat{P} \times \hat{Q}:\varphi(\pi_{\hat{P}}(p)) = \pi_{\hat{Q}}(q)\}.$$
It acts on $\hat{G}$ by $(p,q) \cdot g = pgq^{-1}$. Note that if $\hat{\mathcal{Z}}$ is an algebraic zip datum, then we have an associated \emph{connected algebraic zip datum} $\mathcal{Z} = (G,P,Q,\varphi)$. Its associated zip group $E_{\mathcal{Z}}$ is the identity component of $E_{\hat{\mathcal{Z}}}$, and as such also acts on $\hat{G}$.

\begin{defn}
A \emph{frame} of $\hat{\mathcal{Z}}$ is a tuple $(B,T,g)$ consisting of a Borel subgroup $B$ of $G$, a maximal torus $T$ of $B$, and an element $g \in G$, such that
\begin{itemize}
\setlength\itemsep{0em}
\item $B \subset Q$;
\item $gBg^{-1} \subset P$;
\item $\varphi(\pi_{\hat{P}}(gBg^{-1})) = \pi_{\hat{Q}}(B)$;
\item $\varphi(\pi_{\hat{P}}(gTg^{-1})) = \pi_{\hat{Q}}(T)$.
\end{itemize}
\end{defn}

\begin{prop} \label{framesbestaan}
(See \cite[Proposition 3.7]{pinkwedhornziegler2011}) For every algebraic zip datum $\hat{\mathcal{Z}} = (\hat{G},\hat{P},\hat{Q},\varphi)$, every Borel subgroup $B$ of $G$ contained in $Q$ and every maximal torus $T$ of $B$ there exists an element $g \in G$ such that $(B,T,g)$ is a frame of $\mathcal{Z}$. \qed
\end{prop}

We now fix a frame $(B,T,g)$ of $\hat{\mathcal{Z}}$. Let $\hat{P} = U \rtimes \hat{L}$ and $\hat{Q} = V \rtimes \hat{M}$ be the Levi decompositions of $\hat{P}$ and $\hat{Q}$ with respect to $T$; these exist by Proposition \ref{levi}. In this notation
\begin{equation} \label{egroep}
E_{\hat{\mathcal{Z}}} = \{(ul,v\varphi(l)):u \in U, v \in V, l \in \hat{L}\}.
\end{equation}
Furthermore, let $I$ be the type of the parabolic $P$, and let $J$ be the type of the parabolic $Q$. Let $w \in {}^I\hat{W} \subset \recht{Norm}_{\hat{G}}(T)/T$, and choose a lift $\dot{w} \in \recht{Norm}_{\hat{G}}(T)$. If $H$ is a subgroup of $(g\dot{w})^{-1}L(g\dot{w})$, we may compare it with its image under $\varphi \circ \recht{inn}(g\dot{w})$, viewed as subgroups of $\hat{G}$ via the chosen Levi splitting of $\hat{Q}$. The collection of all such $H$ for which $H = \varphi \circ \recht{inn}(g\dot{w})(H)$ has a unique largest element, namely the subgroup generated by all such subgroups.

\begin{defn} \label{hwdef}
Let $H_{w}$ be the unique largest subgroup of $(g\dot{w})^{-1}L(g\dot{w})$ that satisfies the relation $H_{w} = \varphi \circ \recht{inn}(g\dot{w})(H_{w})$. Let $\varphi_{\dot{w}}\colon H_{w} \rightarrow H_w$ be the isogeny induced by $\varphi \circ \recht{inn}(g\dot{w})$, and let $H_{w}$ act on itself by $h \cdot h' = hh'\varphi_{\dot{w}}(h)^{-1}$.
\end{defn}

Since $\varphi \circ \recht{inn}(g\dot{w})(T) = T$, the group $H_{w}$ does not depend on the choice of $\dot{w}$, even though $\varphi_{\dot{w}}$ does. One of the main results of this section is the following result about certain stabilisers.

\begin{thm} \label{pwz81}
(See \cite[Theorem 8.1]{pinkwedhornziegler2011}) Let $w \in {}^I\hat{W}$ and $h \in H_{w}$. Then the stabiliser $\recht{Stab}_{E_{\mathcal{Z}}}(g\dot{w}h)$ is the semidirect product of a connected unipotent normal subgroup and the subgroup
$$\{(\recht{int}(g\dot{w})(h'), \varphi(\recht{int}(g\dot{w})(h')) : h' \in \recht{Stab}_{H_{w}}(h)\},$$
where the action of $H_{w}$ on itself is given by semilinear conjugation as in Definition \ref{hwdef}. \qed
\end{thm}

\begin{defn}
The algebraic zip datum $\hat{\mathcal{Z}}$ is called \emph{orbitally finite} if for any $w \in {}^I\hat{W}$ the number of fixed points of the endomorphism $\varphi \circ \recht{inn}(g\dot{w})$ of $H_w$ is finite; this does not depend on the choice of $\dot{w}$ (see \cite[Proposition 7.1]{pinkwedhornziegler2011}).
\end{defn}

\begin{thm} \label{pwz75c}
(See \cite[Theorem 7.5c]{pinkwedhornziegler2011}) Suppose $\hat{\mathcal{Z}}$ is orbitally finite. Then for any $w \in {}^I\hat{W}$ the orbit $E_{\mathcal{Z}} \cdot (g\dot{w})$ has dimension $\recht{dim}(P) + \ell_{I,J}(w)$. \qed
\end{thm}

\begin{rem}
Although the proofs of these two theorems carry over from the connected case without much difficulty, we feel compelled to make some comments about what exactly changes in the non-connected case, since the proofs of these theorems require most of the material of \cite{pinkwedhornziegler2011}. The key change is that in \cite[Section 4]{pinkwedhornziegler2011} we allow $x$ to be an element of ${}^I\hat{W}^J$, rather than just ${}^IW^J$; however, one can keep working with the connected algebraic zip datum $\mathcal{Z}$, and define from there a connected algebraic zip datum $\mathcal{Z}_{\dot{x}}$ as in \cite[Construction 4.3]{pinkwedhornziegler2011}. There, one needs the Levi decomposition for non-connected parabolic groups; but this is handled in our Proposition \ref{levi}. The use of non-connected groups does not give any problems in the proofs of most propositions and lemmas in \cite[Section 4--8]{pinkwedhornziegler2011}. In \cite[Proposition 4.8]{pinkwedhornziegler2011}, the term $\ell(x)$ in the formula will now be replaced by $\ell_{I,J}(x)$. The only property of $\ell(x)$ that is used in the proof is that if $x \in {}^IW^J$, then $\ell(x) = \#\{\alpha \in \Phi^+ \backslash \Phi_J : x\alpha \in \Phi^- \backslash \Phi_I\}$. In our case, we have $x \in {}^I\hat{W}^J$, and $\ell_{I,J} \colon {}^I\hat{W}^J \rightarrow \mathbb{Z}_{\geq 0}$ is the extension of $\ell \colon {}^IW^J \rightarrow \mathbb{Z}_{\geq 0}$ that gives the correct formula. Furthermore, in the proof of \cite[Proposition 4.12]{pinkwedhornziegler2011} the assumption $x \in {}^IW^J$ is used, to conclude that $x\Phi^+_J \subset \Phi^+$. However, the same is true for $x \in {}^I\hat{W}^J$: write $x = \omega x'$ with $\omega \in \Omega$ and $x' \in {}^{\omega^{-1}I\omega}W^J$; then $x' \Phi^+_J \subset \Phi^+$, and $\omega\Phi^+ = \Phi^+$, since $\Omega$ acts on the based root system. Finally, the proofs of both \cite[Theorem 7.5c]{pinkwedhornziegler2011} and \cite[Theorem 8.1]{pinkwedhornziegler2011} rest on an induction argument, where the authors use that an element $w \in {}^IW$ can uniquely be written as $w = xw_J$, with $x \in {}^IW^J$, $w_J \in {}^{I_x}W_J$, and $\ell(w) = \ell(x)+\ell(w_J)$. The analogous statement that we need to use is that any $w \in {}^I\hat{W}$ can uniquely be written as $w = xw_J$, with $x \in {}^I\hat{W}^J$, $w_J \in {}^{I_x}W_J$, and $\ell_{I,J}(w) = \ell_{I,J}(x) + \ell(w_J)$, see Remark \ref{opmerking}.2. The proofs of the other lemmas, propositions and theorems work essentially unchanged.
\end{rem}

\section{The zeta function of $\hat{G}\textrm{-}\mathsf{Zip}^{\chi,\Theta}_{\mathbb{F}_q}$}

 We fix $q_0$, $G$, $q$, $\chi$ and $\Theta$ as in section \ref{sectiegzips}. The aim of this section is to calculate the zeta function of the stack $\hat{G}\textrm{-}\mathsf{Zip}^{\chi,\Theta}_{\mathbb{F}_q}$. To the triple $(\hat{G},\chi,\Theta)$ we can associate the algebraic zip datum
$$\hat{\mathcal{Z}} = (\hat{G},\hat{P}_+,\sigma(\hat{P}_-), \sigma|_{\hat{L}}).$$
As in section \ref{algebraiczipdata} the zip group $E_{\mathcal{Z}}$ acts on $\hat{G}_{\mathbb{F}_q}$ by $(p_+,p_-) \cdot g' = p_+ g' p_-^{-1}$.

\begin{lem} \label{orbitallyfinite} (See \cite[Remark 3.9]{pinkwedhornziegler2015}) The algebraic zip data $\hat{\mathcal{Z}}$ is orbitally finite.
\end{lem}

\begin{proof}
Let $w \in {}^I\hat{W}$, and choose a lift $\dot{w}$ of $w$ to $\recht{Norm}_{\hat{G}}(T)$ (for some chosen frame $(B,T,g)$ of $\hat{\mathcal{Z}}$). Then $\varphi_{\dot{w}} = \sigma \circ \recht{inn}(g\dot{w})$ is a $\recht{Fr}_p$-semilinear automorphism of $H_w$, hence it defines a model $\tilde{H}_{\dot{w}}$ of $H_{w}$ over $\mathbb{F}_p$. The fixed points of $\varphi_{\dot{w}}$ in $H_w$ then correspond to $\tilde{H}_{\dot{w}}(\mathbb{F}_p)$, which is a finite set; hence $\hat{\mathcal{Z}}$ is orbitally finite.
\end{proof}

\begin{prop} \label{gzipisquot} (See \cite[Proposition 3.11]{pinkwedhornziegler2015}) There is an isomorphism of $\mathbb{F}_q$-stacks $\hat{G}\textrm{-}\mathsf{Zip}^{\chi,\Theta}_{\mathbb{F}_q} \cong [E_{\hat{\mathcal{Z}}} \backslash \hat{G}_{\mathbb{F}_q}]$. \qed
\end{prop}

\begin{lem}
Let $B \subset \sigma(P_-)$ be a Borel subgroup defined over $\mathbb{F}_q$, and let $T \subset B$ be a maximal torus defined over $\mathbb{F}_q$. Then there exists an element $g \in G(\mathbb{F}_q)$ such that $(B,T,g)$ is a frame of $\hat{Z}$.
\end{lem}

\begin{proof}
Consider the algebraic subset $X = \{g \in G(\bar{\mathbb{F}}_q): \sigma(gBg^{-1}) = B, \sigma(gTg^{-1}) = T\}$ of $G(\bar{\mathbb{F}}_q)$. Since $\recht{Norm}_G(B) \cap \recht{Norm}_G(T) = T$, we see that $X$ forms a $T$-torsor over $\mathbb{F}_q$. By Lang's theorem such a torsor is trivial, hence $X$ has a rational point.
\end{proof}

For the rest of this section we fix a frame $(B,T,g)$ as in the previous Lemma.

\begin{lem} \label{galsetisom}
Choose, for every $w \in \hat{W} = \recht{Norm}_{\hat{G}}(T)/T$, a lift $\dot{w} \in \recht{Norm}_{\hat{G}(\bar{\mathbb{F}}_q)}(T(\bar{\mathbb{F}}_q))$. Then the map
\begin{eqnarray*}
\Xi^{\chi,\Theta} &\rightarrow & E_{\hat{\mathcal{Z}}}(\bar{\mathbb{F}}_q) \backslash \hat{G}(\bar{\mathbb{F}}_q) \\
\Theta \cdot w &\mapsto & E_{\hat{\mathcal{Z}}}(\bar{\mathbb{F}}_q) \cdot g\dot{w}
\end{eqnarray*}
is well-defined, and it is an isomorphism of $\recht{Gal}(\bar{\mathbb{F}}_q/\mathbb{F}_q)$-sets that does not depend on the choices of $w$ and $\dot{w}$.
\end{lem}

\begin{proof}
In \cite[Theorem 10.10]{pinkwedhornziegler2011} it is proven that this map is a well-defined bijection independent of the choices of $w$ and $\dot{w}$. Furthermore, if $\tau$ is an element of $\recht{Gal}(\bar{\mathbb{F}}_q/\mathbb{F}_q)$, then the fact that $T$ and $g$ are defined over $\mathbb{F}_q$ implies that $\tau(\dot{w})$ is a lift of $\tau(w)$ to $\recht{Norm}_{\hat{G}}(T)$; this shows that the map is Galois-equivariant.
\end{proof}

\begin{rem}
The isomorphism above, together with the identification $[[E_{\hat{\mathcal{Z}}} \backslash \hat{G}_{\mathbb{F}_q}](\bar{\mathbb{F}}_q)] \cong {}_{E_{\hat{\mathcal{Z}}}}\hat{G}_{\mathbb{F}_q}$ from Remark \ref{fodopmerking}, gives the natural bijection in Proposition \ref{classificatie}.
\end{rem}

\begin{nota} \label{nota}
Let $\Gamma = \recht{Gal}(\bar{\mathbb{F}}_q/\mathbb{F}_q)$. We define the following functions $a,f\colon {}^I\hat{W} \rightarrow \mathbb{Z}_{\geq 0}$ on ${}^I\hat{W}$ as follows:
\begin{itemize}
\setlength\itemsep{0em}
\item $f(w)$ is the cardinality of the $\Gamma$-orbit of $\Theta \cdot w$ in $\Xi^{\chi,\Theta}$, i.e. $f(w) = \#\{\xi \in \Xi^{\chi,\Theta}: \xi \in \Gamma \cdot (\Theta \cdot w)\}$;
\item $a(w) = \recht{dim}(G/P_+) - \ell_{I,J}(w)$.
\end{itemize}
The fact that $a(w)$ is nonnegative for every $w \in {}^I\hat{W}$ follows from Proposition \ref{rekenopm}.2. The function $f$ is clearly $\Theta$-invariant, and $a$ is $\Theta$-invariant by Lemma \ref{thetainv}. As such these functions can be regarded as functions $a,f\colon \Xi^{\chi,\Theta} \rightarrow \mathbb{Z}_{\geq 0}$ or functions $a,f\colon \Gamma \backslash \Xi^{\chi,\Theta} \rightarrow \mathbb{Z}_{\geq 0}$.
\end{nota}

\begin{prop} \label{rekenopm}
For $\xi \in \Xi^{\chi,\Theta}$, let $\mathcal{Y}_{\xi}$ be the $\hat{G}$-zip over $\bar{\mathbb{F}}_q$ corresponding to $\xi$.
\begin{enumerate}
\item The $\hat{G}$-zip $\mathcal{Y}_{\xi}$ has a model over $\mathbb{F}_{q^v}$ if and only if $v$ is divisible by $f(\xi)$.
\item One has $\recht{dim}(\recht{Aut}(\mathcal{Y}_{\xi})) = a(\xi)$ and the identity component of the group scheme $\recht{Aut}(\mathcal{Y}_{\xi})^\recht{red}$ is unipotent.
\end{enumerate}
\end{prop}
 
\begin{proof} { \ }
\begin{enumerate}
\item This follows directly from Proposition \ref{stacksfod}.
\item Let $w \in {}^I\hat{W}$ be such that $\xi = \Theta \cdot w$. By Remark \ref{fodopmerking}, Lemma \ref{orbitallyfinite} and Theorem \ref{pwz75c} we have
\begin{align*}
\recht{dim}(\recht{Aut}(\mathcal{Y}_{\xi})) &= \recht{dim}(\recht{Stab}_{E_{\hat{\mathcal{Z}}}}(g\dot{w}))\\
&= \recht{dim}(G) - \recht{dim}(E_{\hat{\mathcal{Z}}} \cdot g\dot{w}) \\
&= \recht{dim}(G) - \recht{dim}(P) - \ell_{I,J}(\xi) \\
&= a(\xi).
\end{align*}
Furthermore, by Theorem \ref{pwz81} the identity component of the group scheme $\recht{Aut}(\mathcal{Y}_{\xi})^\recht{red}$ is unipotent. \qedhere
\end{enumerate}
\end{proof}
 
\begin{rem}
The formula $\recht{dim}(\recht{Aut}(\mathcal{Y})) = \recht{dim}(G/P) - \ell_{I,J}(\xi)$ from Proposition \ref{rekenopm}.2 apparently contradicts the proof of \cite[Theorem 3.26]{pinkwedhornziegler2015}. There an extended length function $\ell \colon \hat{W} \rightarrow \mathbb{Z}_{\geq 0}$ is defined by $\ell(w\omega) = \ell(w)$ for $w \in W$, $\omega \in \Omega$. It is stated that the codimension of  $E_{\hat{\mathcal{Z}}} \cdot (g\dot{w})$ in $\hat{G}$ is equal to $\recht{dim}(G/P_+)-\ell(w)$, which has to be equal to $\recht{dim}(\recht{Stab}_{E_{\hat{\mathcal{Z}}}}(g\dot{w}))$ since $
\recht{dim}(G) = \recht{dim}(E_{\mathcal{Z}})$. However, the proof seems to be incorrect (and the Theorem itself as well); the dimension formula should follow from \cite[Theorem 5.11]{pinkwedhornziegler2011}, but that result only holds for the connected case. In the nonconnected case one can construct a counterexample as follows. Let $\hat{G}$ be the example of Remark \ref{opmerking}.3 (over $\mathbb{F}_p$), and consider the cocharacter
\begin{eqnarray*}
\chi\colon \mathbb{G}_{\recht{m}} &\rightarrow & G \\
x & \mapsto & \left(\begin{smallmatrix}x & 0 \\ 0 & x^{-1} \end{smallmatrix}\right)
\end{eqnarray*}
Then $L$ is the diagonal subgroup of $G$, $P_+$ the upper triangular matrices, and $P_- = \sigma(P_-)$ the lower triangular matrices, and we can take $g = \left(\begin{smallmatrix} 0 & 1 \\ -1 & 0\end{smallmatrix}\right)$. Employing the notation of (\ref{egroep}), the stabiliser of $g\omega$ in $E_{\hat{\mathcal{Z}}}$ is then equal to
$$\{(lu_+,\sigma(l)u_-)  \in E_{\hat{\mathcal{Z}}} : lu_+ = g\omega \sigma(l)u_- \omega^{-1} g^{-1}\}.$$
Conjugation by $g$ and $\omega$ both exchange $P_+$ and $P_-$, so in the equation 
$lu_+ = g\omega \sigma(l)u_- \omega^{-1} g^{-1}$ the left hand side is in $P_+$, while the right hand side is in $P_-$. This means that both sides have to be in $L$, hence $u_+ = u_- = 1$, and the equation simplifies to $l = -\sigma(l)$. This has only finitely many solutions, hence
$$\recht{codim}(E_{\mathcal{Z}} \cdot (g\dot{w})) =  \recht{dim}(\recht{Stab}_{E_{\hat{\mathcal{Z}}}}(g\dot{w})) = 0 = \recht{dim}(G/P_+) - \ell_{I,J}(\omega) \neq 1 = \recht{dim}(G/P_+) - \ell(\omega).$$
\end{rem}

\begin{rem} \label{opmnonred}
In general the group scheme $\recht{Aut}(\mathcal{Y}_{\xi})$ will not be reduced; see \cite[Remark 3.1.7]{moonen2004} for the first found instance of this phenomenon, or \cite[Remark 3.35]{pinkwedhornziegler2015} for the general case.
\end{rem}

\begin{proof}[Proof of Theorem \ref{maintheorem1}]
With all the previous results all that is left is a straightforward calculation. As in Proposition \ref{rekenopm} let $\mathcal{Y}_{\xi}$ be the $\hat{G}$-zip over $\bar{\mathbb{F}}_q$ corresponding to $\xi$, for every $\xi \in \Xi^{\chi,\Theta}$. Furthermore, for an integer $v \geq 1$, we set $\Xi^{\chi,\Theta}(\mathbb{F}_{q^v}) := (\Xi^{\chi,\Theta})^{\recht{Gal}(\bar{\mathbb{F}}_q/\mathbb{F}_{q^v})}$. We get

\newlength{\eql}
\settowidth{\eql}{=====}

\begin{align*}
Z(\hat{G}\textrm{-}\mathsf{Zip}_{\mathbb{F}_q}^{\chi,\Theta},t) &\makebox[\eql][c]{$\stackrel{\ref{gzipisquot}}{=}$} Z([E_{\hat{\mathcal{Z}}}\backslash \hat{G}],t) \\
&\makebox[\eql][c]{$\stackrel{\ref{zetaquotient},\ref{rekenopm}.2}{=}$} \recht{exp}\left(\sum_{v \geq 1} \frac{t^v}{v} \sum_{C \in (E_{\hat{\mathcal{Z}}}\backslash\hat{G})(\mathbb{F}_{q^v})} q^{-\recht{dim}(\recht{Aut}(P(C))) \cdot v}\right) \\
&\makebox[\eql][c]{$\stackrel{\ref{galsetisom}}{=}$} \recht{exp}\left(\sum_{v \geq 1} \frac{t^v}{v} \sum_{\xi \in \Xi^{\chi,\Theta}(\mathbb{F}_{q^v})} q^{-\recht{dim}(\recht{Aut}(\mathcal{Y}_{\xi})) \cdot v}\right) \\
&\makebox[\eql][c]{$\stackrel{\ref{rekenopm}.2}{=}$} \recht{exp}\left(\sum_{v \geq 1} \frac{t^v}{v} \sum_{\xi \in \Xi^{\chi,\Theta}(\mathbb{F}_{q^v})} q^{-a(\xi) \cdot v}\right) \\
&\makebox[\eql][c]{$\stackrel{\ref{rekenopm}.1}{=}$} \recht{exp}\left(\sum_{v \geq 1} \frac{t^v}{v} \sum_{\substack{\xi \in \Xi^{\chi,\Theta}:\\ f(\xi) \mid v}} q^{-a(\xi) \cdot v}\right) \\
&\makebox[\eql][c]{$=$} \recht{exp}\left(\sum_{\xi \in \Xi^{\chi,\Theta}} \sum_{\substack{v \geq 1:\\ f(\xi) \mid v}} \frac{1}{v}  (q^{-d(\xi)}t)^v\right) \\
&\makebox[\eql][c]{$=$} \prod_{\xi \in \Xi^{\chi,\Theta}}\recht{exp}\left(\sum_{\substack{v \geq 1}} \frac{1}{f(\xi)v}  (q^{-d(\xi)}t)^{f(\xi)v}\right) \\
&\makebox[\eql][c]{$=$} \prod_{\xi \in \Xi^{\chi,\Theta}} (1-(q^{-a(\xi)}t)^{f(\xi)})^{-\frac{1}{f(\xi)}} \\
&\makebox[\eql][c]{$=$} \prod_{\bar{\xi} \in \Gamma \backslash \Xi^{\chi,\Theta}} \prod_{\xi \in \bar{\xi}} (1-(q^{-a(\xi)}t)^{f(\xi)})^{-\frac{1}{f(\xi)}} \\
&\makebox[\eql][c]{$\stackrel{\ref{nota}}{=}$} \prod_{\bar{\xi} \in \Gamma \backslash \Xi^{\chi,\Theta}}(1-(q^{-a(\bar{\xi})}t)^{f(\bar{\xi})})^{-1}. \qedhere
\end{align*}
\end{proof}

\begin{exa} \label{voorbeeld2}
We continue Example \ref{voorbeeld1}. Here ${}^I\hat{W}$ has four elements. Since $\hat{G}$ is split, the action of $\Gamma$ on ${}^I\hat{W}$ is trivial. Since $\Theta$ was trivial as well, we find that $\Gamma \backslash \Xi^{\chi,\Theta}$ is equal to ${}^I\hat{W}$. Since the Galois action is trivial, we find that $f(\xi) = 1$ for all $\xi \in \Xi^{\chi,\Theta}$. Furthermore one has $\recht{dim}(G/P^+) = 1$. From the table in Example \ref{voorbeeld1}, we find that ${}^I\hat{W}$ has two elements $w$ with $\ell_{I,J}(w) = 0$, and two elements with $\ell_{I,J}(w) = 1$; hence from the definition of $a$ in Notation \ref{nota} we get 
$$Z(\hat{G}\textrm{-}\mathsf{Zip}_{\mathbb{F}_p}^{\chi,\Theta},t) = \frac{1}{(1-t)^2(1-qt)^2}.$$
\end{exa}

\section{Stacks of truncated Barsotti--Tate groups} \label{sectieBT}

The aim of this section is to prove Theorem \ref{maintheorem2}. We fix integers $h > 0$ and $0 \leq d \leq h$, and we want to determine the zeta function of the stack $\mathsf{BT}_{n}^{h,d}$ over $\mathbb{F}_p$ for every integer $n \geq 1$. This turns out to be related to the theory of $\hat{G}$-zips and their moduli stacks.

\begin{nota} \label{truncnot}
For the rest of this section, let $G$ be the reductive group $\recht{GL}_{h, \mathbb{F}_p}$. Let $\chi\colon \mathbb{G}_{\recht{m},\mathbb{F}_p} \rightarrow G$ be a cocharacter that induces the weights $0$ with multiplicity $d$ and weight $1$ with multiplicity $h-d$ on the standard representation of $G$. Employing the notation of sections \ref{weylgroepen} and \ref{sectiegzips}, we see that $W$ is the permutation group on $h$ elements (with trivial Galois action), $S = \{(1 \ \ 2), (2 \ \ 3),...,(h-1 \ \ h)\}$, and $I = S \setminus \{(d \ \ d+1)\}$. Note that $\Theta$ has to be trivial, as we can regard it as a subgroup of $\Omega \cong \pi_0(G)$, which is trivial. Hence $\Xi := \Gamma \backslash \Xi^{\chi,\Theta}$ is equal to ${}^IW$, and the map $a\colon \Xi \rightarrow \mathbb{Z}_{\geq 0}$ from Notation \ref{nota} is given by $a(\xi) = \recht{dim}(G/P_+) - \ell(\xi) = d(h-d)-\ell(\xi)$.
\end{nota}

For general $n$, let $\mathsf{D}^{h,d}_n$ be the stack over $\mathbb{F}_p$ of truncated Dieudonné crystals $D$ of level $n$ that are locally of rank $h$, for which the map $F\colon D \rightarrow D^{(p)}$ has rank $d$ locally (see \cite[Remark 2.4.10]{dejong1995}). Then Dieudonné theory (see \cite[3.3.6 \& 3.3.10]{berthelotbreenmessing2006}) tells us that there is a morphism of stacks over $\mathbb{F}_p$
$$\mathbb{D}_n\colon \mathsf{BT}^{h,d}_{n} \rightarrow \mathsf{D}^{h,d}_n$$
that is an equivalence of categories over perfect fields; hence $Z(\mathsf{BT}^{h,d}_{n},t) = Z(\mathsf{D}^{h,d}_n,t)$. As such, we are interested in the categories $\mathsf{D}^{h,d}_n(\mathbb{F}_q)$. An object in this category is Dieudonné module of level $n$, i.e. a triple $(D,F,V)$ where:
\begin{enumerate}
\item $D$ is a free module of rank $h$ over $\recht{W}_n(\mathbb{F}_q)$, the Witt vectors of length $n$ over $\mathbb{F}_q$;
\item $F$ is a $\sigma$-linear endomorphism of $D$ of rank $d$, where $\sigma$ is the automorphism of $\recht{W}_n(\mathbb{F}_q)$ lifting the automorphism $\recht{Fr}_p$ of $\mathbb{F}_q$;
\item $V$ is a $\sigma^{-1}$-linear endomorphism of $D$ satisfying $FV = VF = p$.
\end{enumerate}

Now fix $h$ and $d$, and choose a (non-truncated) Barsotti--Tate group $\mathcal{G}$ of height $h$ and dimension $d$ over $\mathbb{F}_p$. Let $(D_n,F_n,V_n)$ be the Dieudonné module of $\mathcal{G}[p^n]$, and choose a basis for every $D_n$ in a compatible manner (i.e. the basis of $D_n$ is the image of the basis of $D_{n+1}$ under the natural reduction map $D_{n+1}/p^nD_{n+1} \stackrel{\sim}{\rightarrow} D_n$). Then for every power $q$ of $p$, every element in $\mathsf{D}_n^{h,d}(\mathbb{F}_q)$ is isomorphic to $D_{n,g} := (D_n \otimes_{\mathbb{Z}/p^n\mathbb{Z}} \recht{W}_n(\mathbb{F}_q),gF_n,V_ng^{-1})$ for some $g \in \recht{GL}_h(\recht{W}_n(\mathbb{F}_q))$ (See \cite[2.2.2]{vasiu2008}).\\

For a smooth affine group scheme $G$ over $\recht{Spec}(\recht{W}(\mathbb{F}_p))$, let $\mathbb{W}_n(G)$ be the group scheme over $\recht{Spec}(\mathbb{F}_p)$ defined by $\mathbb{W}_n(G)(R) = G(\recht{W}_n(R))$ (see \cite[2.1.4]{vasiu2008}); it is again smooth and affine. For every $n$ there is a natural reduction morphism $\mathbb{W}_{n+1}(G) \rightarrow \mathbb{W}_n(G)$.

\begin{prop}
Let $\mathcal{D}_n := \mathbb{W}_n(\recht{GL}_h)$. Then there exists a smooth affine group scheme $\mathcal{H}$ over $\mathbb{Z}_p$ and for every $n$ an action of $\mathcal{H}_n := \mathbb{W}_n(\mathcal{H})$ on $\mathcal{D}_n$, compatible with the reduction maps $\mathcal{H}_{n+1} \rightarrow \mathcal{H}_n$ and $\mathcal{D}_{n+1} \rightarrow \mathcal{D}_n$, such that for every power $q$ of $p$, there exists for every $g,g' \in \mathcal{D}_n(\mathbb{F}_q)$ an isomorphism of $\mathbb{F}_q$-group varieties
$$\varphi_{g,g'}\colon\recht{Transp}_{\mathcal{H}_{n,\mathbb{F}_q}}(g,g')^{\recht{red}} \stackrel{\sim}{\rightarrow}  \recht{Isom}(D_{n,g},D_{n,g'})^{\recht{red}}$$
that is compatible with compositions in the sense that for every $g,g',g'' \in \mathcal{D}_n(\mathbb{F}_q)$ the following diagram commutes, where the horizontal maps are the natural composition morphisms:

\begin{center}
\begin{tikzcd}
\recht{Transp}_{\mathcal{H}_{n,\mathbb{F}_q}}(g,g')^{\recht{red}} \times \recht{Transp}_{\mathcal{H}_{n,\mathbb{F}_q}}(g',g'')^{\recht{red}} \arrow[r] \arrow[d,"\varphi_{g,g'} \times \varphi_{g',g''}"] & \recht{Transp}_{\mathcal{H}_{n,\mathbb{F}_q}}(g,g'')^{\recht{red}} \arrow[d,"\varphi_{g,g''}"] \\
\recht{Isom}(D_{n,g},D_{n,g'})^{\recht{red}} \times \recht{Isom}(D_{n,g'},D_{n,g''})^{\recht{red}} \arrow[r] & \recht{Isom}(D_{n,g},D_{n,g''})^{\recht{red}}
\end{tikzcd}
\end{center}

\end{prop}

\begin{proof}
The group $\mathcal{H}$ and the action $\mathcal{H}_n \times \mathcal{D}_n \rightarrow \mathcal{D}_n$ are defined in \cite[2.1.1 \& 2.2]{vasiu2008} over an algebraically closed field $k$ of characteristic $p$, but the definition still makes sense over $\mathbb{F}_p$. The isomorphism of groups $\varphi_{g,g}$ is given on $k$-points in \cite[2.4(b)]{vasiu2008}. The definition of the map there shows that it is algebraic and defined over $\mathbb{F}_p$. Since it is an isomorphism on $\bar{\mathbb{F}}_p$-points, it is an isomorphism of reduced group schemes over $\mathbb{F}_p$. Furthermore, a morphism $\recht{Transp}_{\mathcal{H}_{n,\mathbb{F}_q}}(g,g') \rightarrow  \recht{Isom}(D_{n,g},D_{n,g'})$ is given in the proof of \cite[2.2.1]{vasiu2008}. It is easily seen that this map is compatible with compositions in the sense of the diagram above, and that it is equivariant under the action of $\recht{Stab}_{\mathcal{H}_n}(g)(\bar{\mathbb{F}}_p) \cong \recht{Isom}(D_{n,g})(\bar{\mathbb{F}}_p) $. Since both varieties are torsors under this action, this must be an isomorphism as well.
\end{proof}

\begin{cor} \label{dieudonneequiv}
For every power $q$ of $p$ the categories $\mathsf{D}_n^{h,d}(\mathbb{F}_q)$ and $[\mathcal{H}_n \backslash \mathcal{D}_n](\mathbb{F}_q)$ are equivalent.
\end{cor}

\begin{proof}
For every object $D \in \mathsf{D}_n^{h,d}(\mathbb{F}_q)$ choose a $g_D \in \mathcal{D}_n(\mathbb{F}_q)$ such that $D \cong D_{n,g_D}$. Define a functor

$$E\colon \mathsf{D}_n^{h,d}(\mathbb{F}_q) \rightarrow [\mathcal{H}_n \backslash \mathcal{D}_n](\mathbb{F}_q)$$
that sends a $D$ to the pair $(\mathcal{H}_n, f_D)$, where $f_D\colon \mathcal{H}_n \rightarrow \mathcal{D}_n$ is given by $f_D(h) = h \cdot g_D$. We send an isomorphism from $D$ to $D'$ to the corresponding element of 
$$\recht{Isom}((\mathcal{H}_n, f_D),(\mathcal{H}_n, f_{D'})) = \recht{Transp}_{\mathcal{H}_n(\mathbb{F}_q)}(g_D,g_{D'}).$$ 
This functor is fully faithful and essentially surjective, hence an equivalence of categories.
\end{proof}

By \cite[9.18, 8.3 \& 3.21]{pinkwedhornziegler2015} (and before by \cite{kraft1975} and \cite{moonen2001}) the set of isomorphism classes of Dieudonné modules of level 1 over an algebraically closed field of characteristic $p$ are classified by ${}^IW$, where $I$ and $W$ are as in Notation \ref{truncnot}. For each $w \in {}^IW$, let $\mathsf{D}_{n}^{h,d,w}$ be the substack of $\mathsf{D}_{n}^{h,d}$ consisting of truncated Barsotti--Tate groups of level $n$, locally of rank $h$, and with $F$ locally of rank $d$, that are of type $w$ at all geometric points. Then over fields $k$ of characteristic $p$ one has $\mathsf{D}_{n}^{h,d}(k) = \bigsqcup_{w \in {}^IW} \mathsf{D}_{n}^{h,d,w}(k)$ as categories, hence
$$Z(\mathsf{D}_{n}^{h,d},t) = \prod_{w \in {}^IW} Z(\mathsf{D}_{n}^{h,d},t).$$

From Proposition \ref{rekenopm}.1, or directly from the description in \cite[§5]{kraft1975}, each isomorphism class over $\bar{\mathbb{F}}_p$ has a model over $\mathbb{F}_p$. Let $g_{1,w} \in \mathcal{D}_1(\mathbb{F}_p)$ be such that the isomorphism class of $D_{1,g_{1,w}} \otimes \bar{\mathbb{F}}_p$ corresponds to $w \in {}^IW$. For every $n$, let $\mathcal{D}_{n,w}$ be the preimage of $g_{1,w}$ under the reduction map $\mathcal{D}_n \rightarrow \mathcal{D}_1$. Let $\mathcal{H}_{n,w}$ be the preimage of $\recht{Stab}_{\mathcal{H}_1}(g_{1,w})$ in $\mathcal{H}_n$; then analogous to Corollary \ref{dieudonneequiv} for every power $q$ of $p$ we get an equivalence of categories (see \cite[3.2.3 Lemma 2(b)]{gabbervasiu2012})
$$\mathsf{D}_{n}^{h,d,w}(\mathbb{F}_q) \cong [\mathcal{H}_{n,w}\backslash \mathcal{D}_{n,w}](\mathbb{F}_q).$$

\begin{proof}[Proof of Theorem \ref{maintheorem2}] By the discussion above we see that 
$$Z(\mathsf{BT}^{h,d}_{n},t) = \prod_{w \in {}^IW} Z([\mathcal{H}_{n,w}\backslash \mathcal{D}_{n,w}],t).$$
By \cite[9.18 \&  8.3]{pinkwedhornziegler2015} there is an isomorphism of stacks over $\mathbb{F}_p$
$$\mathsf{D}_{1,p}^{h,d} \stackrel{\sim}{\rightarrow} G\textrm{-}\mathsf{Zip}^{\chi,\Theta}_{\mathbb{F}_p},$$
where $G,\chi,\Theta$ are as in Notation \ref{truncnot}. By Proposition \ref{rekenopm}.2, or earlier by \cite[2.1.2(i) \& 2.2.6]{moonen2004}, the group scheme $\recht{Stab}_{\mathcal{H}_1}(g_w)^{\recht{red}} \cong \recht{Aut}(D_{1,g_1})^{\recht{red}}$ has a identity component that is unipotent of dimension $a(w)$. The reduction morphism $\mathcal{H}_n \rightarrow \mathcal{H}_1$ is surjective and its kernel is unipotent of dimension $h^2(n-1)$, see \cite[3.1.1 \& 3.1.3]{gabbervasiu2012}. This implies that $\mathcal{H}_{n,w}$ has a unipotent identity component of dimension $h^2(n-1)+a(w)$. Now fix a $g_{n,w} \in \mathcal{D}_{n,w}(\mathbb{F}_p)$; then we can identify $\mathcal{D}_{n,w}$ with the affine group $X = \mathbb{W}_{n-1}(\recht{Mat}_{h \times h})$, by sending an $x \in X$ to $g_{n,w}+ps(x)$, where $s\colon \mathbb{W}_{n-1}(\recht{Mat}_{h \times h}) \stackrel{\sim}{\rightarrow} p\mathbb{W}_{n}(\recht{Mat}_{h \times h}) \subset \mathbb{W}_{n}(\recht{Mat}_{h \times h})$ is the canonical identification. Furthermore, the action of an element $h \in \mathcal{H}_{n,w}$ on $(g_{n,w}+ps(x)) \in \mathcal{D}_{n,w}$ is given by $f(h)(g_{n,w}+ps(x))f'(h)$ for some $f,f'\colon \mathcal{H}_{n,w} \rightarrow \mathbb{W}_n(\recht{GL}_h)$ (see  \cite[2.2.1a]{vasiu2008}). From this we see that the induced action of $\mathcal{H}_{n,w}$ on the variety $X$ is given by
$$h \cdot x = f(h)xf'(h) + \frac{1}{p}(f(h)g_{n,w}f'(h)-g_{n,w}),$$
which makes sense because $f(h)g_{n,w}f'(h)$ is equal to $g_{n,w}$ modulo $p$. If we regard $X$ as $\mathbb{W}_{n-1}(\mathbb{G}_{a}^{h^2})$ via its canonical coordinates, this shows us that the action of $\mathcal{H}_{n,w}$ on $X$ factors through the canonical action of $\mathbb{W}_{n-1}(\mathbb{G}_{a}^{h^2}) \rtimes \mathbb{W}_{n-1}(\recht{GL}_{h^2})$ on $\mathbb{W}_{n-1}(\mathbb{G}_{a}^{h^2})$. This algebraic group is connected, so we can apply Theorem \ref{affienquotient}, from which we find
$$Z([\mathcal{H}_{n,w}\backslash \mathcal{D}_{n,w}],t) = \frac{1}{1-p^{\recht{dim}(\mathcal{D}_{n,w})-\recht{dim}(\mathcal{H}_{g,w})}} = \frac{1}{1-p^{h^2(n-1)- (h^2(n-1)+a(w))}t} = \frac{1}{1-p^{-a(w)}t},$$
which completes the proof.
\end{proof}

\begin{rem}
Since the zeta function $Z(\mathsf{BT}_{n}^{h,d},t)$ does not depend on $n$, one might be tempted to think that the stack $\mathsf{BT}^{h,d}$ of non-truncated Barsotti--Tate groups of height $h$ and dimension $d$ has the same zeta function. However, this stack is not of finite type. For instance, every Barsotti--Tate group $\mathcal{G}$ over $\mathbb{F}_q$ has a natural injection $\mathbb{Z}_p^{\times} \hookrightarrow \recht{Aut}(\mathcal{G})$, which shows us that its zeta function is not well-defined.
\end{rem}

\DeclareRobustCommand{\VAN}[3]{#3}

\printbibliography
\end{document}